\documentclass[leqno]{siamltex}
\usepackage{amssymb}
\usepackage{amsbsy}
\usepackage{newlfont}
\usepackage{amsmath}
\usepackage{psfrag}
\usepackage{graphicx}
\usepackage{color}  
\usepackage{xspace,float,stackrel,enumerate,subfigure}
\usepackage{bm}
\usepackage{bbm}
\usepackage{mathrsfs} 
\usepackage{remarks}
\usepackage{mydef}
\let\citep=\cite

\begin{document}
\newcommand\footnotemarkfromtitle[1]{%
\renewcommand{\thefootnote}{\fnsymbol{footnote}}%
\footnotemark[#1]%
\renewcommand{\thefootnote}{\arabic{footnote}}}

\title{Invariant domains and first-order\\ continuous finite element 
approximation \\ for hyperbolic systems\footnotemark[1]}

\author{Jean-Luc Guermond\footnotemark[2]
\and Bojan Popov\footnotemark[2]}
\date{Draft version \today}

\maketitle

\renewcommand{\thefootnote}{\fnsymbol{footnote}} \footnotetext[1]{
  This material is based upon work supported in part by the National
  Science Foundation grants DMS-1217262, by the Air Force Office of
  Scientific Research, USAF, under grant/contract number
  FA99550-12-0358, and by the Army Research  Office under grant/contract
  number W911NF-15-1-0517.  Draft version, \today}
\footnotetext[2]{Department of Mathematics, Texas A\&M University 3368
  TAMU, College Station, TX 77843, USA.}
\renewcommand{\thefootnote}{\arabic{footnote}}

\begin{abstract}
  We propose a numerical method to solve general hyperbolic systems in
  any space dimension using forward Euler time stepping and continuous
  finite elements on non-uniform grids.  The properties of the method
  are based on the introduction of an artificial dissipation that is
  defined so that any convex invariant sets containing the initial
  data is an invariant domain for the method.  The invariant domain
  property is proved for any hyperbolic system provided a CFL
  condition holds. The solution is also shown to satisfy a discrete
  entropy inequality for every admissible entropy of the system. The
  method is formally first-order accurate in space and can be made
  high-order in time by using Strong Stability Preserving
  algorithms. This technique extends to continuous finite elements the
  work of \cite{Hoff_1979,Hoff_1985}, and \cite{Frid_2001}.
  \end{abstract}



\pagestyle{myheadings} \thispagestyle{plain} \markboth{J.L. GUERMOND,
  B. POPOV}{Invariant domains and $\calC^0$ finite element 
approximation of hyperbolic systems}

\section{Introduction} \label{Sec:introduction}
The objective of this paper is to investigate a first-order
approximation technique for nonlinear hyperbolic systems using continuous
finite elements and explicit time stepping on non-uniform meshes. Consider the 
following hyperbolic system in conservation form
\begin{equation}
 \label{def:hyperbolic_system} 
 \begin{cases} \partial_t \bu + \DIV \bef(\bu)=0, 
\quad \mbox{for}\, (\bx,t)\in \Real^d\CROSS\Real_+.\\
\bu(\bx,0) = \bu_0(\bx), \quad \mbox{for}\, \bx\in \Real^d.
\end{cases}
\end{equation}
where the dependent variable $\bu$ takes values in $\Real^m$ and the
flux $\bef$ takes values in $(\Real^m)^d$. In this paper $\bu$ is
considered as a column vector $\bu=(u_1,\ldots,u_m)\tr$. The flux is a
matrix with entries $f_{ij}(\bu)$, $1\le i\le m$, $1\le j\le d$ and
$\DIV\bef$ is a column vector with entries
$(\DIV\bef)_i= \sum_{1\le j\le d}\partial_{x_j} f_{ij}$. For any
$\bn=(n_1\ldots,n_d)\tr\in \Real^d$, we denote $\bef(\bu)\SCAL\bn$ the
column vector with entries $\sum_{1\le l\le d} n_lf_{il}(\bu)$, where
$i\in\intset{1}{m}$.  The unit sphere in $\Real^d$ centered at $0$ is
denoted by $S^{d-1}(\bzero,1)$.

To simplify questions regarding boundary conditions, we assume that
either periodic boundary conditions are enforced, or the initial data
is compactly supported or  constant outside a compact set.  
In both cases we denote by $\Dom$ the spatial
domain where the approximation is constructed. The domain $\Dom$ is
the $d$-torus in the case of periodic boundary conditions. In the case
of the Cauchy problem, $\Dom$ is a compact, polygonal portion of
$\Real^d$ large enough so that the domain of influence of $\bu_0$ is
always included in $\Dom$ over the entire duration of the simulation.

The method that we propose is explicit in time and uses continuous
finite elements on non-uniform grids in any space dimension. The
algorithm is described in \S\ref{Sec:The_scheme}, see
\eqref{def_of_scheme_dij} with definitions
\eqref{def_of_mi}-\eqref{def_of_cij}-\eqref{Def_of_dij}.  It is a
somewhat loose adaptation of the non-staggered Lax-Friedrichs scheme
to continuous finite elements.  The key results of the paper are
Theorem~\ref{Thm:UL_is_invariant} and Theorem~\ref{Thm:entrop_ineq}.
It is shown in Theorem~\ref{Thm:UL_is_invariant} that the proposed
scheme preserves all the convex invariant sets as defined in
Definition~\ref {Def:invariant_set} and it is shown in
Theorem~\ref{Thm:entrop_ineq} that the approximate solution satisfies
a discrete entropy inequality for every entropy pair of the hyperbolic
system.  Similar results have been established for various finite
volumes schemes by \cite{Hoff_1979,Hoff_1985},
\cite{Perthame_Shu_1996}, \cite{Frid_2001} for the compressible Euler
equations and the p-system. Our scheme has no restriction on the
nature of the hyperbolic system, besides the speed of propagation
being finite. To the best of our knowledge, we are not aware of any
similar scheme in the continuous finite element literature.

The paper is organized as follows. The notions of invariant sets and
invariant domains with various examples and other preliminaries are
introduced in Section~\ref{Sec:preliminaries}.  The method is
introduces in Section~\ref{Sec:first_order_scheme}. Stability
properties of the algorithm are analyzed in
Section~\ref{Sec:stability}. Numerical illustrations and comparisons
with existing first-order methods are presented in
Section~\ref{Sec:numerical}.

\section{Preliminaries}
\label{Sec:preliminaries}
The objective of this section is to introduce notation and preliminary
results that will be useful in the rest of the paper.  We mostly use
the notation and the terminology of
\cite{Chueh_Conley_Smoller,Hoff_1979,Hoff_1985,Frid_2001}.  The reader
who is familiar with the notions of invariant domains and Riemann
problems may skip this section and go directly to
\S\ref{Sec:first_order_scheme}, although the reader should be aware that our definitions of invariant
sets and domains are slightly different from those of \citep{Chueh_Conley_Smoller,Hoff_1979,Hoff_1985,Frid_2001}.

\subsection{Riemann problem}
We assume that \eqref{def:hyperbolic_system} is such that there is a
clear notion for the solution of the Riemann problem.  That is to say
there exists an (nonempty) admissible set $\calA \subset\Real^{m} $
such that for any pair of states $(\bu_L,\bu_R)\in \calA\CROSS \calA$
and any unit vector $\bn\in S^{d-1}(\bzero,1)$, the following one-dimensional
Riemann problem
\begin{equation}
 \label{def:Riemann_problem} 
  \partial_t \bu + \partial_x (\bef(\bu)\SCAL\bn)=0, 
\quad  (x,t)\in \Real\CROSS\Real_+,\qquad 
\bu(x,0) = \begin{cases} \bu_L, & \text{if $x<0$} \\ \bu_R,  & \text{if $x>0$}, \end{cases}
\end{equation}
has a unique (physical) solution, which we
henceforth denote $\bu(\bn,\bu_L,\bu_R)$.

The theory of the Riemann problem for general nonlinear hyperbolic systems with
data far apart is an open problem. Moreover, it is unrealistic to expect a general theory
for any system with arbitrary initial data. However, when the system is
strictly hyperbolic with smooth flux and all the characteristic fields
are either genuinely nonlinear or linearly degenerate, it is possible
to show that there exists $\delta>0$ such that the Riemann problem has
a unique self-similar weak solution in Lax's form for any initial data such that
$\|\bu_L-\bu_R\|_{\ell^2}\le \delta$, see \cite{Lax_1957_II} and
\cite[Thm~5.3]{Bressan_2000}. In particular there are $2m$ numbers
\begin{equation}
\lambda_1^-\le \lambda_1^+ \le  \lambda_2^-\le \lambda_2^+ \le \ldots \le 
\lambda_m^-\le \lambda_m^+
\end{equation}
defining up to $2m+1$ sectors (some could be empty) in the $(x,t)$ plane:
\begin{equation}
\frac{x}{t}\in (-\infty,\lambda_1^-), \quad
\frac{x}{t}\in (\lambda_1^-,\lambda_1^+), \ldots, \quad
\frac{x}{t}\in (\lambda_m^-,\lambda_m^+), \quad \frac{x}{t}\in (\lambda_m^+,\infty). 
\end{equation}
The Riemann solution is $\bu_L$ in the sector
$\frac{x}{t}\in (-\infty,\lambda_1^-)$ and $\bu_R$ in the last sector
$\frac{x}{t}\in (\lambda_m^+,\infty)$.  The solution in the other
sectors is either a constant state or an expansion, see
\cite[Chap.~5]{Bressan_2000}.  The sector
$\lambda_1^- t < x < \lambda_m^+ t$, $0<t$, is henceforth referred to as the
Riemann fan.  The key result that we are going to use is that there is
a maximum speed of propagation
$\lambda_{\max}(\bn,\bu_L,\bu_R):=\max(|\lambda_1^-|,|\lambda_m^+|)$
such that for $t\ge 0$ we have
\begin{equation}
\bu(x,t)= \begin{cases}
\bu_L, & \text{if $x \le -t \lambda_{\max}(\bn,\bu_L,\bu_R)$}\\
\bu_R, & \text{if $x \ge   t \lambda_{\max}(\bn,\bu_L,\bu_R)$}.
\end{cases} \label{finite_speed}
\end{equation}

Actually, even if the above structure of the Riemann solution is not
available or valid, we henceforth make the following assumption:
\begin{equation}
\begin{aligned}
  &\text{The unique solution of \eqref{def:Riemann_problem} has a
    finite speed of propagation for any
    $\bn$,}\\[-3pt]
  &\text{\ie there is $\lambda_{\max}(\bn,\bu_L,\bu_R)$ such that
    \eqref{finite_speed} holds. }
\end{aligned} \label{finite_speed_assumption}
\end{equation}
For instance, this is the case for
strictly hyperbolic systems that may have characteristic families that
are either not genuinely nonlinear or not linearly degenerate, see \eg
\cite[Thm.1.2]{Liu_1975} and \cite[Thm.~9.5.1]{Dafermos_2000}.  We
refer to \cite[Thm.~1]{Osher_1983} for the theory of the Riemann
problem for scalar conservation equations with nonconvex fluxes. In
the case of general hyperbolic systems, we refer to \cite[Section
14]{Bianchini_Bressan_2005} for characterizations of the Riemann
solution using viscosity regularization. We also refer to
\cite[Thm.~2]{Young_2002} for the theory of the Riemann problem for
the $p$-system with arbitrary data (\ie with possible formation of
vacuum).

The following elementary result is an important, well-known,
consequence of \eqref{finite_speed}, \ie  the Riemann solution is equal to $\bu_L$
for $x\in (-\infty,\lambda_1^-t)$ and equal $\bu_R$
for $x\in (\lambda_m^+ t,\infty)$:
\begin{lemma} \label{Lem:elemetary_Riemann_pb} Let
  $\bu_L,\bu_R \in \calA$, let $\bu(\bn,\bu_L,\bu_R)$ be the Riemann
  solution to \eqref{def:Riemann_problem}, let
  $ \overline\bu(t,\bn,\bu_L,\bu_R) :=\int_{-\frac12}^{\frac12}
  \bu(\bn,\bu_L,\bu_R)(x,t) \diff x $
  and assume that $t\, \lambda_{\max}(\bn,\bu_L,\bu_R) \le \frac12$,
  then
\begin{equation}
  \overline\bu(t,\bn,\bu_L,\bu_R)  = \frac{1}{2}(\bu_L+\bu_R) 
  - t\big(\bef(\bu_R)\SCAL\bn - \bef(\bu_L)\SCAL\bn\big).
\label{elemetary_Riemann_pb}
\end{equation}
\end{lemma}%

If the system \eqref{def:hyperbolic_system} has an entropy pair $(\eta,\bq)$,
and if the Riemann solution is defined to be 
entropy satisfying, \ie if the following holds
\begin{equation}
\partial_t\eta(\bu(\bn,\bu_L,\bu_R)) 
+ \partial_x \big(\bq(\bu(\bn,\bu_L,\bu_R))\SCAL\bn\big) \le 0, 
\label{entropy_inequality_elemetary_Riemann_pb}
\end{equation}
in some appropriate sense (distribution sense, measure sense, \etc),
then we have the following additional result.
\begin{lemma} \label{Lem:entropy_elemetary_Riemann_pb} Let
  $(\eta,\bq)$ be an entropy pair for \eqref{def:hyperbolic_system}
  and assume that \eqref{entropy_inequality_elemetary_Riemann_pb}
  holds. Let $\bu_L,\bu_R \in \calA$ and let $\bu(\bn,\bu_L,\bu_R)$ be
  the Riemann solution to \eqref{def:Riemann_problem}. Assume that
  $t\, \lambda_{\max}(\bn,\bu_L,\bu_R) \le \frac12$, Then
\begin{equation}
  \eta(\overline\bu(t,\bn,\bu_L,\bu_R) )
\le \tfrac{1}{2}(\eta(\bu_L)+\eta(\bu_R)) 
  - t(\bq(\bu_R)\SCAL\bn - \bq(\bu_L)\SCAL\bn).
\label{entropy_elemetary_Riemann_pb}
\end{equation}
\end{lemma}
\begin{proof} Under the CFL assumption $t\,
  \lambda_{\max}(\bn,\bu_L,\bu_R) \le \frac12$, the inequality
  \eqref{entropy_inequality_elemetary_Riemann_pb} implies that
\begin{equation}
\int_{-\frac12}^{\frac12}\eta(\bu(\bn,\bu_L,\bu_R))(x,t)\diff x 
\le \tfrac{1}{2}(\eta(\bu_L)+\eta(\bu_R)) 
-t(\bq(\bu_R)\SCAL\bn - \bq(\bu_L)\SCAL\bn). \label{def_of_entropy_inequality}
\end{equation}
Jensen's inequality
$\eta(\overline\bu(t,\bn,\bu_L,\bu_R)) \le
\int_{-\frac12}^{\frac12}\eta(\bu(\bn,\bu_L,\bu_R)(x,t))\diff x $
then implies the desired result.
\end{proof}
\subsection{Invariant sets and domains}
We introduce in this section the notions of invariant sets and
invariant domains. Our definitions are slightly different from those
in \cite{Chueh_Conley_Smoller,Hoff_1985,Smoller_1983,Frid_2001}.
We will associate invariant sets only with solutions of
Riemann problems and define invariant domains only for an
approximation process.

\begin{definition}[Invariant set] \label{Def:invariant_set} We say that a set
  $A\subset \calA\subset \Real^m$ is invariant for
  \eqref{def:hyperbolic_system} if for any pair $(\bu_L,\bu_R)\in A\CROSS A$,
 any unit vector $\bn\in \calS^{d-1}(\bzero,1)$, and any $t>0$, the average of the entropy
  solution of the Riemann problem \eqref{def:Riemann_problem} over the
  Riemann fan, say,
$\frac{1}{t(\lambda_m^+-\lambda_1^-)}  
\int_{\lambda_1^-t}^{\lambda_m^+t} \bu(\bn,\bu_L,\bu_R)(x,t) \diff x$,
 remains in $A$.
\end{definition}

Note that, the above definition implies that given $t>0$ and any
interval $I$ such that $(\lambda_1^-t,\lambda_m^+t)\subset I$, we have
that $\frac{1}{I} \int_{I} \bu(\bn,\bu_L,\bu_R)(x,t) \diff x\in A$. Note also that  most of the time 
expansion waves and shocks are not invariant sets.

We now introduce the notion of invariant domain for an approximation process.
 Let $\bX_h \subset L^1(\Real^d;\Real^m)$ be a finite-dimensional
approximation space and let
$S_h : \bX_h \ni \bu_h\longmapsto S_h(\bu_h)\in \bX_h$ be a discrete
process over $\bX_h$. Henceforth we abuse the language by
saying that a member of $\bX_h$, say $\bu_h$, is in the set
$A\subset \Real^m$ when actually
we mean that $\{\bu_h(\bx) \st \bx \in\Real \}\subset A$.
\begin{definition}[Invariant domain]\label{Def:invariant_domain}
  A convex invariant set $A \subset \calA\subset \Real^m$ is said to
  be an invariant domain for the process $S_h$ if and only if for any
  state $\bu_h$ in $A$, the state $S_h(\bu_h)$ is also in $A$.
\end{definition}

For scalar conservation equations the notions of invariant sets and
invariant domains are closely related to the maximum principle, see
Example~\ref{Ex:1}. In the case of nonlinear systems, the notion of maximum
principle does not apply and must be replaced by the notion of
invariant domain. To the best of our knowledge, the definition of
invariant sets for the Riemann problem was introduced in
\cite{Nishida_1968}, and the general theory of positively invariant
regions was developed in \cite{Chueh_Conley_Smoller}. Applications and
extensions to numerical methods were developed in
\cite{Hoff_1979,Hoff_1985} and \cite{Frid_2001}. 

The invariant domain theory when $m=2$ and $d=1$
relies on the existence of global Riemann
invariants; the best known examples are the hyperbolic systems of
isentropic gas dynamics in Eulerian and Lagrangian form, see
Example~\ref{Ex:2} and \cite{Lions_Perthame_Souganidis_1996}.
For results on general hyperbolic systems, we refer to \cite{Frid_2001},
where a characterization of invariant domains for the Lax-Friedrichs
scheme and some flux splitting schemes is given. In particular the existence
of invariant domains is established for the above mentioned schemes for
the compressible Euler equations in the general 
case $m = d+2$ (positive density, internal energy, and minimum principle on the
specific entropy), see \cite[Thm. 7 and Thm. 8]{Frid_2001}.
Similar results have been established for various finite volume
schemes in two-space dimension for the Euler equations
in \cite[Thm.~3]{Perthame_Shu_1996}.

The objective of this paper is to propose an explicit numerical method
based on continuous finite elements to approximate
\eqref{def:hyperbolic_system} such that any convex invariant set of
\eqref{def:hyperbolic_system} is an invariant domain for the
process generated by the said numerical method.

To facilitate the reading of the paper we now
illustrate the abstract notions of invariant sets and invariant
domains with some examples.

\subsection{Example 1: scalar equations} 
\label{Ex:1} Assume that $m=1$ and $d$ is arbitrary, \ie
\eqref{def:hyperbolic_system} is a scalar conservation
equation. Provided $\bef\in \text{Lip}(\Real;\Real^d)$, any bounded
interval is an admissible set for \eqref{def:hyperbolic_system}. For
any Riemann data $u_L,u_R$, the maximum speed of propagation in
\eqref{finite_speed} is bounded by
$\lambda_{\max}(u_L,u_R):=\|\bef\SCAL\bn\|_{\text{Lip}(u_{\min},u_{\max})}$
where $u_{\min}=\min(u_L,u_R)$, $u_{\max}=\max(u_L,u_R)$. If $\bef$ is
convex and is of class $C^1$, we have
$\lambda_{\max}(u_L,u_R) = \max(|\bn\SCAL\bef'(u_L)|,|\bn\SCAL\bef'(u_R)|)$ if
$\bn\SCAL\bef'(u_L)\le \bn\SCAL\bef'(u_R)$ and
$\lambda_{\max}(u_L,u_R) =\bn\SCAL (\bef(u_L)-\bef(u_R))/(u_L-u_R)$ otherwise.  Any
interval $[a,b]\subset\Real$ is admissible and is an invariant set for
\eqref{def:hyperbolic_system}, \ie if $u_R,u_L\in [a,b]$, then
$a\le u(\bn,u_L,u_R) \le b$ for all times; this is the maximum
principle.  For any $a\le b\in \Real$, the interval $[a,b]$ is an
invariant domain for any maximum principle satisfying numerical
scheme. Note that the maximum principle can be established for a large
number of numerical methods (whether monotone or not), see for example
\cite{Crandall_Majda}.

\subsection{Example 2: p-system}
\label{Ex:2} The one-dimensional motion of an isentropic gas is
modeled by the so-called $p$-system, and in Lagrangian coordinates the
system is written as follows:
\begin{equation}
 \label{def:p_system} 
\begin{cases} \partial_t v+\partial_x u=0,\\
\partial_t u+\partial_x p(v)=0, \quad \mbox{for}\,\, (x,t)\in \Real\CROSS\Real_+.
\end{cases}
\end{equation}
Here $d=1$ and $m=2$.
The dependent variables are the velocity $u$ and the
specific volume $v$, \ie the reciprocal of density.  
The mapping $v\mapsto p(v)$
is the pressure and is assumed to be of class $C^2(\Real_+;\Real)$ and to
satisfy
\begin{equation}
 \label{def:pressure} 
p'<0, \qquad 0<p''.
\end{equation}
A typical example is the so-called gamma-law, $p(v)=r v^{-\gamma}$,
where $r>0$ and $\gamma \ge 1$.  Using the notation $\bu=(v,u)\tr$,
any set $\calA$ in  $(0,\infty)\CROSS \Real$ is admissible.

Using the notation
$\diff\mu:=\sqrt{-p'(s)}\,\diff s$, and assuming
$\int_1^\infty \diff \mu<\infty$, the system has two families of
global Riemann invariants:
\begin{equation} \label{def:Riemann_inv}
  w_1(\bu)=u+\int_v^\infty\!\!\!\! \diff \mu, \quad\mbox{and}\quad
  w_2(\bu)=u-\int_v^\infty\!\!\!\! \diff \mu.
\end{equation}
Note that 
$\int_1^\infty\!\! \diff \mu<\infty$ if 
$\gamma>1$.  If $\gamma=1$ we can use
$w_1(\bu)=u-\sqrt{r}\log v$ and $w_2(\bu)=u+\sqrt{r}\log v$.
Let $a,b\in\Real$, then it can be shown that any
set $A_{ab}\in \Real_+\CROSS \Real$ of the form
\begin{equation}\label{speed_upper_bound_p_system}
A_{ab} := \{\bu\in \Real_+\CROSS \Real \st a \le w_2(\bu),
\ w_1(\bu) \le b\}
\end{equation}
is an invariant set for the system \eqref{def:p_system} for $\gamma\ge
1$, see \cite[Exp.~3.5, p.~597]{Hoff_1985} for a proof in the context
of parabolic regularization, or use the results from \cite{Young_2002}
for a direct proof. Moreover, $A_{ab}$ is an invariant domain for the
Lax-Friedrichs scheme, see \cite[Thm.~2.1]{Hoff_1979} and
\cite[Thm.~4.1]{Hoff_1985}.

Since in the rest of the paper the maximum wave speed is the only
information we are going to need from the Riemann solution, we give
the following result.
\begin{lemma} \label{Lem:lambda_max_p_system}
  Let $(v_L,u_L), (v_R,u_R) \in \Real_+\CROSS \Real$ with
  $v_R,v_L<\infty$.  Then
\begin{equation*}
\lambda_{\max}(\bu_L,\!\bu_R) = \!\!
\begin{cases}
  \sqrt{-p'(\min(v_L,v_R))},
  & 
\text{if $u_L-u_R >\! \sqrt{(v_L-v_R)(p(v_R)-p(v_L))}$,} \\
  \sqrt{-p'(v^*)},
  & \text{ otherwise,}
\end{cases}
\end{equation*}
where $v^*$ is the unique solution of
$\phi(v):=f_L(v)+f_R(v)+u_L-u_R=0$ and
\begin{equation*}
f_Z(v) :=\begin{cases} -\sqrt{(p(v)-p(v_Z)(v_Z-v)},\quad \mbox{if}\,\,v\le v_Z\\
\displaystyle\int_{v_Z}^v \diff \mu, \quad \mbox{if}\,\, v>v_Z.
\end{cases}
\end{equation*}
Upon setting $w_1^{\max}:=\max(w_1(\bu_L),w_1(\bu_R))$ and
$w_2^{\min}:=\min(w_2(\bu_L),w_2(\bu_R))$ we have also have $v^0\le \min(v_L,v_R, v^*)$,
\ie $\lambda_{\max}(\bu_L,\!\bu_R)\le \sqrt{-p'(v^0)}$, where 
\[
v^0:=(\gamma r)^{\frac{1}{\gamma-1}}
\left(\frac{4}{(\gamma-1)(w_1^{\max}-w_2^{\min})}\right)^{\frac{2}{(\gamma-1)}}.
\]
\end{lemma}
\begin{proof} It is well know that the solution of the Riemann problem
  consists of three constant states $\bu_L$, $\bu^*$, and $\bu_R$
  connected by two waves: a 1-wave connects $\bu_L$ and $\bu^*$, and a
  2-wave connects $\bu^*$ and $\bu_R$. Moreover, a vacuum forms if and
  only if $\lim_{v\to+\infty}\phi(v)\ge 0$, see \cite{Young_2002} for
  details. In the presence of vacuum the equation $\phi(v)=0$ has no
  solutions and in this case we conventionally set $v^*:=+\infty$ and
  $\sqrt{-p'(v^*)}:=0$. Note that since $\phi$ is an increasing and concave
  up function with $\lim_{v\to 0+}\phi(v)=-\infty$, the
  solution $v^*$ is unique. We also have that the maximum speed of the
  exact solution is $\lambda_{\max}(\bu_L,\!\bu_R)=\max(
  \sqrt{-p'(v_L)},\!\sqrt{-p'(v^*)},\!\sqrt{-p'(v_R)})$. The only
  possibility for $\lambda_{\max}(\bu_L,\!\bu_R)=\sqrt{-p'(v^*)}$ is
  if $v^*\le \min(v_L,v_R)$, \ie the solution contains two shock waves
  which is equivalent to $\phi(\min(v_L,v_R))\ge 0$. Using the
  definition of $\phi$ we derive that
  $\lambda_{\max}(\bu_L,\!\bu_R)=\sqrt{-p'(v^*)}$ if and only if
  $\phi(\min(v_L,v_R))=u_L-u_R -\sqrt{(v_L-v_R)(p(v_R)-p(v_L))}\ge
  0$. This finishes the proof of the first part of the lemma.

  The exact value of $v^*$ can be found using Newton's method starting
  with a guess $v^0\le v^*$.  This guarantees that at each step of
  Newton's method the estimated maximum speed is an
  upper bound for the exact maximum speed. One can obtain such a guess $v^0$
  by using the invariant domain property
  \eqref{speed_upper_bound_p_system}, \ie we define the state
  $\bu^0:=(v^0,u^0)$ by $w_1^{\max} = w_1(\bu^0)$ and
  $w_2^{\min} = w_2(\bu^0)$ thereby giving
\[
v^0=(\gamma r)^{\frac{1}{\gamma-1}}
\left(\frac{4}{(\gamma-1)(w_1^{\max}-w_2^{\min})}\right)^{\frac{2}{(\gamma-1)}}.
\]
The invariant domain property guarantees that $v^0\le v^*$.  Hence,
the result is established.
\end{proof}

\begin{remark}
  Note that the estimate on $\lambda_{\max}(\bu_L,\bu_R)$ given in
  Lemma~\ref{Lem:lambda_max_p_system} is valid whether vacuum is
  created or not in the Riemann solution.
\end{remark}

\begin{remark}
We only consider the case where both $\bu_L$ and $\bu_R$, are
not vacuum states in Lemma~\ref{Lem:lambda_max_p_system}, 
since the algorithm that we propose in this
paper never produces vacuum states if vacuum is not present in the
initial data. 
\end{remark}

\subsection{Example 3: Euler} \label{Ex:3} Consider the
compressible Euler equations
\begin{equation}
\partial_t \bc +\DIV(\bef(\bc)) =0, \quad 
\bc=\left(\begin{matrix}
\rho \\
\bbm \\ 
E
\end{matrix}\right),\qquad
\bef(\bc)=
\left(
\begin{matrix}
\bbm \\ 
 \bbm {\otimes} \frac{\bbm}{\rho} + p\polI \\ 
\frac{\bbm}{\rho} (E+p) 
\end{matrix}\right),
\label{Euler_equations}
\end{equation}
where the independent variables are the density $\rho$, the momentum
vector field $\bbm$ and the total energy $E$.  The velocity vector
field $\bu$ is defined by $\bu:=\bbm/\rho$ and the internal energy
density $e$ by $e:=E-\frac12|\bu|^2$. The quantity $p$ is the
pressure.  The symbol $\polI$ denotes the identity matrix in
$\polR^d$.  Let $s$ be the specific entropy of the system, and assume
that $-s(e,\rho^{-1})$ is strictly convex. It is known that
\begin{equation}
A_r:= \{ (\rho,\bbm,\rho E) \st \rho\ge0,e\ge0,s\ge r\}
\end{equation}
is an invariant set for the Euler system for any $r\in \Real$. It is
shown in \cite[Thm. 7 and 8]{Frid_2001} that the set $A_r$ is convex
and is an invariant domain for the Lax-Friedrichs scheme. 

Let $\bn \in S^{d-1}(\bzero,1)$ and let us
formulate the Riemann problem~\eqref{def:Riemann_problem} for the
Euler equations. This problem was first described in the context of
dimension splitting schemes with $d=2$ in
\cite[p.~526]{Chorin_1976}. The general case is treated in
\cite[p.~188]{Colella_1990}, see also \cite[Chapter 4.8]{Toro_2009}.
We make a change of basis and introduce $\bt_1,\ldots,\bt_{d-1}$ so
that $\{\bn,\bt_1,\ldots,\bt_{d-1}\}$ forms an orthonormal basis of
$\polR^d$. With this new basis we have $\bbm=(m, \bbm^{\perp})\tr$,
where $m:=\rho u$, $u:=\bu\SCAL\bn$, $\bbm^{\perp}\!:=\rho(\bu\SCAL
\bt_1,\ldots, \bu\SCAL\bt_{d-1}):=\rho\bu^\perp$. The projected
equations are
\begin{equation}
\partial_t \bc +\partial_{x}(\bn\SCAL\bef(\bc)) =\mathbf{0}, \quad 
\bc=\left(\begin{matrix}
\rho \\
m\\
\bbm^{\perp} \\ 
E
\end{matrix}\right),\qquad
\bn\SCAL\bef(\bc)=
\left(
\begin{matrix}
m\\
\tfrac{1}{\rho}m^2 +p \\ 
u \bbm^{\perp} \\ 
u (E+ p)
\end{matrix}\right).
\label{Euler_projected}
\end{equation}
Using the density $\rho$ and the specific entropy $s$ as dependent
variables for the pressure, $p(\rho,s)$, the
linearized Jacobian is 
\[
\left(\begin{matrix} 
u                        & \rho  & \mathbf{0}\tr  & 0 \\
\rho^{-1}\partial_\rho p   & u     & \mathbf{0}\tr & \rho^{-1}\partial_s p \\
\mathbf{0}             & \mathbf{0}&  u \polI   & \mathbf{0} \\  
0                      & 0 &     \mathbf{0}\tr   & u   
\end{matrix}\right).
\]
The eigenvalues  are $u$, with multiplicity $d$,
$u+\sqrt{\partial_\rho p(\rho,s)}$, with multiplicity 1, and
$u-\sqrt{\partial_\rho p(\rho,s)}$, with multiplicity 1. One key observation
is that the Jacobian does not depend on $\bbm^\perp$, see \cite[p.~150]{Toro_2009}.
As a consequence the solution of the Riemann problem with data $(\bc_L,\bc_R)$,
 is such that $(\rho,u,p)$ is obtained
as the solution to the one-dimensional Riemann problem 
\begin{equation}
\partial_t\left(\begin{matrix}
\rho \\
m\\
\mathcal E
\end{matrix}\right)+ \partial_x
\left(
\begin{matrix}
m\\
\tfrac{1}{\rho}m^2 +p \\ 
u (\mathcal{E}+ p)
\end{matrix}\right)=0, \quad \text{with} \quad e=\mathcal{E} - \frac{m^2}{2\rho}
\label{1D_Euler}
\end{equation}
with data $\bc_L^\bn:=(\rho_L,\bbm_L\SCAL\bn,\mathcal{E}_L)$,
$\bc_R^\bn:=(\rho_R,\bbm_R\SCAL\bn,\mathcal{E}_R)$, where $\mathcal{E}_Z = E_Z -
\frac12 \frac{\|\bbm_Z^\perp\|_{\ell^2}^2}{\rho_Z}$, $Z\in \{L,R\}$.  Moreover, for an
ideal gas obeying the caloric equation of state $p=(\gamma-1)\rho e$,
it can be shown (see \cite[p.~150]{Toro_2009}) that $\bbm^\perp$ is
the solution of the transport problem $\partial_t \bbm^\perp +
\partial_x (u\bbm) =0$. The bottom line of this argumentation is that
the maximum wave speed in \eqref{Euler_projected} is
\[
\lambda_{\max}(\bc_L,\bc_R)=
\max(|\lambda_1^-(\bc_L^\bn,\bc_R^\bn)|,|\lambda_3^+(\bc_L^\bn,\bc_R^\bn)|).
\]
where $\lambda_1^-(\bc_L^\bn,\bc_R^\bn)$ and
$\lambda_3^+(\bc_L^\bn,\bc_R^\bn)$ are the two extreme wave speeds in
the  Riemann problem~\eqref{1D_Euler} with data
$(\bc_L^\bn,\bc_R^\bn)$.

We now determine the values of $\lambda_1^-(\bc_L^\bn,\bc_R^\bn)$ and
$\lambda_3^+(\bc_L^\bn,\bc_R^\bn)$.  We only consider the case where both
states, $\bc_L$ and $\bc_R$, are not vacuum states, since the
algorithm that we are proposing in this paper never produces vacuum
states if vacuum is not present in the initial data.  That is, we
assume $\rho_L,\rho_R> 0$ and $p_L,p_R\ge 0$. Then the local sound speed
is given by $a_Z=\sqrt{\frac{\gamma p_Z}{\rho_Z}}$ where $Z$ is either
$L$ or $R$.  We introduce the following notations
$A_Z:=\frac{2}{(\gamma+1)\rho_Z}$, $B_Z:=\frac{\gamma-1}{\gamma+1}p_Z$
and the functions
\begin{align}
\phi(p)&:=f(p,L)+f(p,R)+u_R-u_L\\
f(p,Z)&:=\begin{cases} (p-p_{Z})\left(\frac{A_{Z}}{p+B_{Z}}\right)^{\frac12} 
& \text{if $p\ge p_{Z}$},\\
\frac{2a_{Z}}{\gamma-1}\left(\left(\frac{p}{p_{Z}}\right)^{\frac{\gamma-1}{2\gamma}}-1\right) 
& \text{if  $p < p_{Z}$},
\end{cases}
\end{align}
where again $Z$ is either $L$ or $R$.  It is shown in
\cite[Chapter~4.3.1]{Toro_2009} that the function
$\phi(p)\in C^1( \Real_+;\Real)$ is monotone increasing
and concave down. 
Observe that $\phi(0) = u_R-u_L-\frac{2a_L}{\gamma-1}-\frac{2a_R}{\gamma-1}$.
Therefore, $\phi$ has a unique positive root if and
only if the non-vacuum condition
\begin{equation}\label{no_vacuum}
u_R-u_L<\frac{2a_L}{\gamma-1}+\frac{2a_R}{\gamma-1}
\end{equation}
holds, see \cite[(4.40), p.~127]{Toro_2009}; we denote this root by
$p^*$, \ie $\phi(p^*)=0$ and $p^*$ can be found via Newton's method.
If \eqref{no_vacuum} does not hold we set $p^*=0$.  Then it can be
shown that, whether there is formation of vacuum or not, we have
\begin{align}
\lambda_1^-(\bc_L^\bn,\bc_R^\bn)=u_L - a_L\left(
1+\frac{\gamma+1}{2\gamma}\left(\frac{p^*-p_L}{p_L}\right)_+
\right)^\frac12, \\
\lambda_3^+(\bc_L^\bn,\bc_R^\bn)=u_R + a_R\left(
1+\frac{\gamma+1}{2\gamma}\left(\frac{p^*-p_R}{p_R}\right)_+
\right)^\frac12 ,
\label{estimate_speeds_Euler}
\end{align}
where $z_+:=\max(0,z)$. 

\begin{remark}[Fast algorithm] Note that if both $\phi(p_L)>0$ and $\phi(p_R)>0$,
  there is no need to compute $p^*$, since in this case
  $\lambda_1^-(u_L,u_R)=u_L - a_L$ and
  $\lambda_3^+(u_L,u_R)=u_R + a_R$, 
 \ie two
  rarefaction waves are present in the solution with a possible formation of vacuum. This observation
  is important since traditional techniques to compute $p^*$ may
  require a large number of iterations in this situation, see
  \cite[p.~128]{Toro_2009}. Note finally that there is no need to
  compute $p^*$ exactly since one needs only an upper bound on
  $\lambda_{\max}$. A very fast algorithm, with guaranteed upper bound
  on $\lambda_{\max}$ up to any prescribed accuracy $\epsilon$ of the
  type
  $\lambda_{\max}\le \tilde \lambda_{\max}\le (1+\epsilon) \lambda_{\max}$,
  is described in \cite{Guermond_Popov_Fast_Riemann_2015}.

\end{remark}

\section{First order method}
\label{Sec:first_order_scheme} 
We describe in this section an explicit first-order finite element
technique that, up to a CFL restriction, preserves all  convex
invariant sets of \eqref{def:hyperbolic_system} that contain
reasonable approximations of $\bu_0$.  Although most of the arguments
invoked in this section are quite standard and mimic Lax's
one-dimensional finite volume scheme, we are not aware of the
existence of such a finite-element-based scheme in the literature.

\subsection{The finite element space}
We want to approximate the solution of \eqref{def:hyperbolic_system}
with continuous finite elements.  Let $\famTh$ be a shape-regular
sequence of affine matching meshes. The elements in the mesh sequence
are assumed to be generated from a finite number of reference elements
denoted $\wK_1,\dots,\wK_\varpi$.  For example, the mesh $\calT_h$
could be composed of a combination of triangles and parallelograms in
two space dimensions ($\varpi=2$ in this case); it could also be
composed of a combination of tetrahedra, parallelepipeds, and
triangular prisms in three space dimensions ($\varpi=3$ in this
case). The affine diffeomorphism mapping $\wK_r$ to an arbitrary
element $K\in \calT_h$ is denoted $T_K : \wK_r \longrightarrow K$ and
its Jacobian matrix is denoted $\polJ_K$, $1\le r\le \varpi$. We now
introduce a set of reference Lagrange finite elements
$\{(\wK_r,\wP_r,\wSigma_r)\}_{1\le r\le \varpi}$ (the index
$r\in\intset {1}{\varpi}$ will be omitted in the rest of the paper to
alleviate the notation). Then we define the scalar-valued and
vector-valued Lagrange finite element spaces
\begin{align} \label{eq:Xh}
P(\calT_h) &=\{ v\in \calC^0(\Dom;\Real)\st 
v_{|K}{\circ}T_K \in \wP,\ \forall K\in \calT_h\},\qquad 
\bP(\calT_h) = [P(\calT_h)]^m.
\end{align}  
where $\wP$ is the reference polynomial space defined on $\wK$ (note
that the index $r$ has been omitted). Denoting $\nf:=\dim\wP$ and
denoting by $\{\wba_i\}_{i\in\intset{1}{\nf}}$ the Lagrange nodes of
$\wK$, we assume that the space $\wP$ is such that
\begin{equation}
  \min_{1\le \ell \le \nf} \wv(\wba_\ell) \le \wv(\wbx) \le 
  \max_{1\le \ell \le \nf} \wv(\wba_\ell), 
  \quad \forall \wv\in \wP, \forall \wbx\in \wK. \label{reference_mesh_convexity_assumption} 
\end{equation}
Denoting by $\polP_1$ and $\polQ_1$ the set of multivariate
polynomials of total and partial degree at most $1$, respectively;
the above assumption holds for $\wP=\polP_1$ when $K$ is a
simplex and $\wP=\polQ_1$ when $K$ is a parallelogram or a
cuboid. This assumption holds also for first-order prismatic elements
in three space dimensions.

Let $\{\ba_i\}_{i\in\intset{1}{\Nglob}}$ be the collection of all the
Lagrange nodes in the mesh $\calT_h$, and let
$\{\varphi_i\}_{i\in\intset{1}{\Nglob}}$ be the corresponding global
shape functions.  Recall that $\{\varphi_i\}_{i\in\intset{1}{\Nglob}}$
forms a basis of $P(\calT_h)$ and $\varphi_i(\ba_j)=\delta_{ij}$. The
Lagrange interpolation operator in $\calP(\calT_h)$ is denoted
$\Pi_h:\calC^0(\overline\Dom) \longrightarrow \calP(\calT_h)$.  Recall that
$\Pi_h(v) = \sum_{1\le i \le \Nglob} v(\ba_i) \varphi_i$. We denote by
$S_i$ the support of $\varphi_i$ and by $\mes{S_i}$ the measure of
$S_i$, $i\in\intset{1}{\Nglob}$. We also define $S_{ij}:=S_i\cap S_j$
the intersection of the two supports $S_i$ and $S_j$. Let $E$ be a
union of cells in $\calT_h$; we define
$\calI(E):=\{j\in\intset{1}{\Nglob} \st \mes{S_j \cap E}\not=0\}$ the
set that contains the indices of all the shape functions whose support
on $E$ is of nonzero measure. We are going to regularly invoke
$\calI(K)$ and $\calI(S_i)$ and the partition of unity property:
$\sum_{i\in\calI(K)} \varphi_i(\bx) =1$ for all $\bx\in K$.

We define the operator $\opCoord: P(\calT_h) \longrightarrow
\Real^\Nglob$ so that $\opCoord(v_h)$ is the coordinate vector of
$v_h$ in the basis $\{\varphi_i\}_{i\in\intset{1}{\Nglob}}$, \ie
$v_h=\sum_{i=1}^\Nglob \opCoord(v_h)_i\varphi_i$. Note that
$\opCoord(v_h)_i=v_h(\ba_i)$.  We are also going to use capital
letters for the coordinate vectors to alleviate the notation; for
instance we shall write $\sfV=\opCoord(v_h)$ when the context is
unambiguous.  Note finally that the above assumptions on the mesh and
the reference elements imply the following property: for all $\bx\in
K$ and all $K\in \calT_h$,
\begin{equation}
  \min_{\ell \in\calI(K)}\opCoord(v_h)_\ell \le v_h(\bx) \le 
  \max_{\ell \in\calI(K)} \opCoord(v_h)_\ell, 
  \quad \forall v_h\in P(\calT_h). \label{mesh_convexity_assumption} 
\end{equation}
We define similarly the  $\bopCoord: \bP(\calT_h) \longrightarrow
(\Real^m)^\Nglob$, \ie $\bv_h = \sum_{i=1}^I \bopCoord(\bv_h)_i \varphi_i$,
or equivalently $\bopCoord(\bv_h)_i = \bv_h(\ba_i)$.

Let $\calM\in \Real^{\Nglob\CROSS \Nglob}$ be the
consistent mass matrix with entries
$\int_{S_{ij}} \varphi_i(\bx)\varphi_j(\bx)\diff x$, and let $\calM^L$
be the diagonal lumped mass matrix with entries
\begin{equation}
  m_i:= \int_{S_i} \varphi_i(\bx)\diff x. \label{def_of_mi}
\end{equation}
The partition of unity property implies that $m_i = \sum_{j\in
  \calI(S_i)} \int \varphi_j(\bx)\varphi_i(\bx)\diff x$, \ie the
entries of $\calM^L$ are obtained by summing the rows of $\calM$.

\subsection{The scheme} \label{Sec:The_scheme} Let $\bu_{h0} \in \calP(\calT_h)$ be a
reasonable approximation of $\bu_0$ (we shall be more precise in the
following sections). Let $n\in \polN$, $\dt$ be the time step, $t^n$
be the current time, and let us set $t^{n+1}=t^n+\dt$. Let
$\bu_h^n\in\calP(\calT_h)$ be the space approximation of $\bu$ at time $t^n$
and set $\bsfU^n=\bopCoord(\bu_h^n)$.  We propose to compute
$\bu_h^{n+1}$ by
 \begin{equation}
m_i \frac{\bsfU_i^{n+1}-\bsfU_i^n}{\dt} 
+ \int_\Dom  \DIV(\Pi_h\bef(\bu_h^n)) \varphi_i \diff x  \\
 - \sum_{j\in \calI(S_i)} d_{ij} \bsfU^n_j  = 0,  \label{def_of_scheme_dij}
\end{equation}
where $\bsfU^{n+1}=\bopCoord(\bu_h^{n+1})$ and the
lumped mass matrix is used for the approximation of the time
derivative.
The coefficient $d_{ij}$ is an artificial
viscosity for the pair $(i,j)$ that has yet to be clearly identified. 
For the time
being we assume that 
\begin{equation}
  d_{ij}\ge 0, \quad \text{if}\ \ \  i\not= j,\quad
  d_{ij}=d_{ji},\quad  \text{and} \quad  d_{ii}:=\sum_{i\ne j\in\calI(S_i)}  -d_{ji}.
\label{introduction_dij}
\end{equation}  
Using that
$\DIV(\Pi_h\bef(\bu_h^n)) = \sum_j \bef(\bsfU_j^n)\SCAL\GRAD
\varphi_j$, the above equation simplifies into
\begin{equation}
  m_i \frac{\bsfU_i^{n+1}-\bsfU_i^n}{\dt} 
  + \sum_{j\in \calI(S_i)} \bef(\bsfU_j^n)\SCAL \bc_{ij}
  - \bsfU^n_j d_{ij} = 0,  \label{update_cijdij_dij}
\end{equation}
where the coefficients $\bc_{ij}\in \Real^d$ 
are defined by
\begin{equation}
\bc_{ij} = \int_\Dom \varphi_i \GRAD\varphi_j \diff x,\
\label{def_of_cij}
\end{equation}

\begin{remark}[Conservation]
The definition $d_{ii}:=\sum_{i\ne j\in\calI(S_i)} -d_{ji}$ implies that
$\sum_{j\in\calI(S_i)} d_{ji}=0$, which in turn implies conservation,
\ie
$\int_\Dom \bu_h^{n+1} \diff x =\int_\Dom \bu_h^{n} \diff x +
\int_\Dom \DIV(\Pi_h\bef(\bu_h^n))\diff x$.
Note also that the symmetry assumption in \eqref{introduction_dij} implies
$d_{ii}:=\sum_{i\ne j\in\calI(S_i)} -d_{ij}$, which is often easier to compute.  
\end{remark}

\subsection{The convex combination argument}
\label{Sec:convex_argument}
We motivate the choice of the artificial viscosity coefficients
$d_{ij}$ in this section. Observing that the partition of unity
property $\sum_{j\in \calI(S_i)} \varphi_j =1$ and
\eqref{introduction_dij} imply conservation, \ie
\begin{equation}
\sum_{j\in\calI(S_i)} \bc_{ij} = 0, \qquad  \sum_{j\in\calI(S_i)} d_{ij} =0.
\end{equation} 
we re-write \eqref{update_cijdij_dij}  as follows: 
\begin{equation}
  m_i \frac{\bsfU_i^{n+1}-\bsfU_i^n}{\dt} 
  = -\sum_{j\in \calI(S_i)} (\bef(\bsfU_j^n) - \bef(\bsfU_i^n)) \SCAL \bc_{ij}
  + d_{ij}(\bsfU^n_j+\bsfU^n_i).  
\label{Unplusone_convex_combination_dij_rewritten}
\end{equation}
Using again conservation, \ie
$d_{ii} = -\sum_{i\ne j\in\calI(S_i)} d_{ij}$, we finally arrive at
\begin{equation}
  \bsfU_i^{n+1}  = \bsfU_i^n
  \Big(1 - \sum_{i\ne j\in \calI(S_i)} \frac{2\dt d_{ij}}{m_i} \Big) 
  + \sum_{i\ne j\in \calI(S_i)}\frac{2\dt d_{ij}}{m_i}\overline\bsfU_{ij}^{n+1}.   
\label{Unplusone_convex_combination_dij}
\end{equation}
where we have introduced the auxiliary quantities
\begin{equation}
\overline\bsfU_{ij}^{n+1} := \frac12(\bsfU^n_j+\bsfU^n_i) 
- (\bef(\bsfU_j^n) - \bef(\bsfU_i^n)) \SCAL \frac{\bc_{ij}}{2d_{ij}}. \label{def:Ubar_dij}
\end{equation}

A first key observation is that \eqref{Unplusone_convex_combination_dij}
is a convex combination provided $\dt$ is small enough.  A second key
observation at this point is that upon setting $\bn_{ij}:=
\bc_{ij}/\|\bc_{ij}\|_{\ell^2}$, $\overline\bsfU_{ij}^{n+1}$ is exactly
of the form $\overline\bu(t,\bn_{ij},\bsfU_i,\bsfU_j)$ as defined in
\eqref{elemetary_Riemann_pb} with a fake time
$t=\|\bc_{ij}\|_{\ell^2}/2d_{ij}$. The CFL condition
$t\lambda_{\max}(\bn_{ij},\bu_L,\bu_R)\le \frac12 $ in
Lemma~\ref{Lem:elemetary_Riemann_pb} motivates the following
definition for the viscosity coefficients $d_{ij}$
\begin{equation}
  d_{ij} := \max(\lambda_{\max}(\bn_{ij},\bsfU_i^n,\bsfU_j^n)
  \|\bc_{ij}\|_{\ell^2},
\lambda_{\max}(\bn_{ji},\bsfU_j^n,\bsfU_i^n) \|\bc_{ji}\|_{\ell^2}), \label{Def_of_dij} 
\end{equation}
where recall that $\lambda_{\max}(\bn_{ij},\bsfU_i,\bsfU_j)$ is defined
in the assumption \eqref{finite_speed_assumption}.

\begin{remark}[Symmetry]
  If either $\ba_i$ or $\ba_j$ is an interior node in the mesh, one
  integration by parts implies that $\bc_{ij} = -\bc_{ji}$, which in
  turn implies
  $\lambda_{\max}(\bn_{ij},\bsfU_i,\bsfU_j)=
  \lambda_{\max}(\bn_{ji},\bsfU_j,\bsfU_i)$.
  In conclusion
  $\lambda_{\max}(\bn_{ij},\bsfU_i,\bsfU_j) \|\bc_{ij}\|_{\ell^2}=
  \lambda_{\max}(\bn_{ji},\bsfU_j,\bsfU_i) \|\bc_{ji}\|_{\ell^2}$
  if either $\ba_i$ or $\ba_j$ is an interior node.
\end{remark}

\begin{remark}[Upwinding]
  Note that in the scalar one-dimensional case when the flux $f$ is
  linear, \eqref{def_of_scheme_dij} gives the usual upwinding
  first-order method.
 \end{remark}

\section{Stability analysis}
\label{Sec:stability}
We analyze the stability properties of the scheme
\eqref{def_of_scheme_dij} with the viscosity defined in
\eqref{Def_of_dij}.

\subsection{Invariant domain property} 
Upon defining $h_K:=\text{diam}(K)$, 
the global maximum mesh size is denoted $h = \max_{K\in\calT_h} h_K$.
The local minimum mesh size, $\hmin_K$, for any $K\in \calT_h$ is
defined as follows:
\begin{equation} \label{eq:local_mesh}
\hmin_K:=\frac{1}{\max_{i\ne j\in \calI(K)}\|\GRAD\varphi_i\|_{\bL^\infty(S_{ij})}},
\end{equation}  
and the global minimum mesh size is
$\hmin:=\min_{K\in \calT_h} \hmin_K$.  Due to the shape regularity
assumption, the quantities $\hmin_K$ and $h_K$ are uniformly
equivalent, but it will turn out that using $\hmin_K$ instead of $h_K$
gives a sharper estimate of the CFL number.
Let $\nf:=\text{card}(\calI(K))$ and let us define
$\vartheta_K:=\frac{1}{\nf-1}$. Note that
\begin{equation}
  0< \vartheta_{\min} := \min_{\famTh} \min_{K\in \calT_h} 
  \vartheta_K<+\infty,\hspace{-5pt}\label{Def_of_rho}
\end{equation}
since there are at most $\varpi$ reference elements defining the mesh sequence.
We also  introduce the mesh-dependent
quantities
\begin{equation}
\mu_{\min} := \min_{K\in\calT_h}\min_{i\in
  \calI(K)} \frac{1}{|K|}\int_K \varphi_i(\bx)\diff x,\quad 
\mu_{\max} :=
\max_{K\in\calT_h}\max_{i\in \calI(K)} \frac{1}{|K|}\int_K
\varphi_i(\bx)\diff x.
\end{equation}  
Note that $\mu_{\min} =\mu_{\max}= \frac{1}{\nf} = \frac{1}{d+1}$ for
meshes uniquely composed of simplices and
$\mu_{\min} =\mu_{\max}=2^{-d}$ for meshes uniquely composed of
parallelograms and cuboids.  We now prove the main result of the
paper.
\begin{theorem} \label{Thm:UL_is_invariant} Let $A\subset \calA$ be an
  invariant set for \eqref{def:hyperbolic_system} in the sense of
  Definition~\ref{Def:invariant_set}. Assume that $A$ is convex and 
\begin{equation} \lambda_{\max}(A):=\max_{\bn\in
      S^{d-1}(\bzero,1)}\max_{\bu_L,\bu_R\in A}
    \lambda_{\max}(\bn,\bu_L,\bu_R)<\infty, \label{lambda_max_bounded}
\end{equation}
where $S^{d-1}(\bzero,1)$ is the unit sphere in $\Real^d$. Assume that
  $\bu_{h0} \in A$ and $\dt$ is such that
\begin{equation}
2\dt\frac{\lambda_{\max}(A)}{\hmin}
\frac{\mu_{\max}}{\mu_{\min}\vartheta_{\min}}\le 1.
\label{CFL_assumption}
\end{equation}
Then
\begin{enumerate}[(i)]
\item \label{Lem:UL_is_invariant:item1} $A$ is an invariant domain
for the solution process $\bu_h^n \longmapsto \bu_h^{n+1}$ for all $n\ge 0$.
\item \label{Lem:UL_is_invariant:item2} Given $n\ge 0$ and
  $i\in\intset{1}{\Nglob}$, let $B\subset A$ be a convex invariant set
  such that $\bsfU_l^n\in B$ for all $l\in \calI(S_i)$, then
$\bsfU_i^{n+1} \in B$.
\end{enumerate}
\end{theorem}

\begin{proof}
  We prove the statement \eqref{Lem:UL_is_invariant:item1} by induction. Assume that
  $\bu_h^n\in A$ for some $n\ge 0$; we are going to prove that
  $\bu_h^{n+1}\in A$. Note that $\bu_{h0} \in A$ by assumption.  Let
  $i\in\intset{1}{\Nglob}$ and consider the update
  \eqref{def_of_scheme_dij} rewritten in the form
  \eqref{Unplusone_convex_combination_dij}.  Observe that upon defining
  $\bn_{ij} := \bc_{ij}/\|\bc_{ij}\|_{\ell^2}$, the quantity
  $\overline\bsfU_{ij}^{n+1}$ defined in \eqref{def:Ubar_dij} is exactly of
  the form $\overline\bu(t,\bn_{ij},\bsfU_{i}^n,\bsfU_{j}^n)$ as defined
  in \eqref{elemetary_Riemann_pb} with the flux $\bef\SCAL \bn_{ij}$
  and the fake time $t=\|\bc_{ij}\|_{\ell^2}/2d_{ij}$. The definition
\begin{align}
d_{ij} \ge \lambda_{\max}(\bn_{ij},\bsfU_{i}^n,\bsfU_{j}^n) \|\bc_{ij}\|_{\ell^2},  
\label{lower_bound_on_dij}
\end{align}
is the CFL condition for the conclusions of Lemma~\ref{Lem:elemetary_Riemann_pb}
to hold with fake time $t=\|\bc_{ij}\|_{\ell^2}/2d_{ij}$.  Since $A$ is a convex
invariant set we have
$\overline\bsfU_{ij}^{n+1}:=\overline\bu(t,\bn_{ij},\bsfU_{i}^n,\bsfU_{j}^n)
\in A$ for all $j\in \calI(S_i)$.
Let us now prove that \eqref{Unplusone_convex_combination_dij} is indeed a
convex combination by proving that $1 - \sum_{i\ne j\in \calI(S_i)}
\frac{2\dt d_{ij}}{m_i} = 1 + \frac{2\dt d_{ii}}{m_i}\ge 0 $. 
Note first that 
\[
\|\bc_{ij}\|_{\ell^2} \le \int_{S_{ij}}\|\GRAD\varphi_j\|_{\ell^2} \varphi_i \diff x 
\le \hmin^{-1} \int_{S_{ij}}\varphi_i \diff x 
\le \hmin^{-1} \mu_{\max} \mes{S_{ij}}.
\]
The definition of $d_{ii}$ implies that
\begin{align*}
-d_{ii} & \le \frac{\lambda_{\max}(A)}{\hmin} \mu_{\max} \sum_{i\not= j\in\calI(S_i)}\mes{S_{ij}}
\le \frac{\lambda_{\max}(A)}{\hmin}\frac{\mu_{\max}}{\vartheta_{\min}}\mes{S_{i}} .
\end{align*}
The using that $\mu_{\min} \mes{S_i} \le m_i$, we infer that
\[
-2\dt\frac{d_{ii}}{m_i} \le 2\dt\frac{\lambda_{\max}(A)}{\hmin}
\frac{\mu_{\max}}{\mu_{\min}\vartheta_{\min}}\le 1,
\]
which proves the result owing to the CFL assumption
\eqref{CFL_assumption}.  Hence \eqref{Unplusone_convex_combination_dij}
defines $\bsfU_i^{n+1}$ as a convex combination between $\bsfU_i^n$ and
the collection of states $\{\overline\bsfU_{ij}^{n+1}\}_{j\in
  \calI(S_i)}$.  The convexity of $A$ implies that $\bsfU_i^{n+1}\in
A$, since $\bsfU_i^{n}\in A$ by assumption and we have established
above that $\overline\bsfU_{ij}^{n+1}\in A$ for all $j\in \calI(S_i)$.
The space approximation being piecewise linear, a convexity
argument implies again that $\bu_h^{n+1} \in A$, which proves the induction 
assumption.

Note in passing that we have also proved the following local
invariance property: given any convex invariant set $B\subset A$ that
contains $\{\bsfU_l^n\}_{l\in \calI(S_i)}$, $\bsfU_i^{n+1}$ is also in
$B$, \ie the local statement \eqref{Lem:UL_is_invariant:item2} holds. This
completes the proof.
\end{proof}

\begin{remark}
  The arguments invoking the convex combination~\eqref{Unplusone_convex_combination_dij} and
  the one-dimensional Riemann averages~\eqref{def:Ubar_dij} are
  similar in spirit to those used in the proof of Theorem~3 in
  \cite{Perthame_Shu_1996}.
\end{remark}

\subsection{Discrete entropy inequality} 
We now derive a local entropy inequality.
\begin{theorem}\label{Thm:entrop_ineq} 
  Let $A\subset \calA$ be a convex invariant set for
  \eqref{def:hyperbolic_system}. Let $(\eta,\bq)$ be an entropy pair
  for \eqref{def:hyperbolic_system}.  Assume that
  \eqref{entropy_inequality_elemetary_Riemann_pb} holds for any
  Riemann data $(\bu_L,\bu_R)$ in $A$, and any $\bn\in
  S^{d-1}(\bzero,1)$. Assume also that \eqref{lambda_max_bounded} and
  \eqref{CFL_assumption} hold, then we have the following for any
  $n\ge 0$ and any $i\in \intset{1}{I}$:
\begin{align}
  \frac{m_i}{\dt} (\eta(\bsfU_i^{n+1}) - \eta(\bsfU_i^n))   
   + \int_\Dom \DIV(\Pi_h\bq(\bu_h^n)) \varphi_i \diff x  +
   \sum_{i\ne j\in \calI(S_i)} d_{ij} \eta(\bsfU_{j}^{n}) \le 0. 
\label{discrete_entropy_inequality}
\end{align}
\end{theorem}
\begin{proof}
  Let $(\eta,\bq)$ be an entropy pair for the system
  \eqref{def:hyperbolic_system}.  Let $i\in\intset{1}{\Nglob}$, then
  recalling \eqref{Unplusone_convex_combination_dij}, the CFL condition
  and the convexity of $\eta$ imply that
\[
  \eta(\bsfU_i^{n+1})  \le 
  \Big(1 - \sum_{i\ne j\in \calI(S_i)} \frac{2\dt d_{ij}}{m_i} \Big)\eta(\bsfU_i^n) 
  + \sum_{i\ne j\in \calI(S_i)}\frac{2\dt d_{ij}}{m_i}\eta(\overline\bsfU_{ij}^{n+1}).   
\]
Owing to Lemma~\ref{Lem:entropy_elemetary_Riemann_pb} we have
\[
 \eta(\overline\bsfU^{n+1}_{ij})
\le \tfrac{1}{2}(\eta(\bsfU_i^n)+\eta(\bsfU_j^n)) 
  - t(\bq(\bsfU_j^n)\SCAL\bn_{ij} - \bq(\bsfU_i^n)\SCAL\bn_{ij}).
\]
with $t=\|\bc_{ij}\|_{\ell^2}/2 d_{ij}$; hence,
\begin{multline*}
 \frac{m_i}{\dt} (\eta(\bsfU_i^{n+1}) - \eta(\bsfU_i^n))   \le 
   \sum_{i\ne j\in \calI(S_i)} 2 d_{ij}(\eta(\overline\bsfU_{ij}^{n+1}) -\eta(\bsfU_i^n)) \\
\le  \sum_{i\ne j\in \calI(S_i)} d_{ij}(\eta(\bsfU_{j}^{n}) - \eta(\bsfU_i^n))
- \|\bc_{ij}\|_{\ell^2}(\bq(\bsfU_j^n)\SCAL\bn_{ij} - \bq(\bsfU_i^n)\SCAL\bn_{ij}).
\end{multline*}
The conclusion follows from the definitions of $\bn_{ij}$, $\bc_{ij}$
and $d_{ij}$.
\end{proof}
\begin{remark}
  One recovers the
  equation \eqref{def_of_scheme_dij} from 
\eqref{discrete_entropy_inequality} with $\eta(\bv) = \bv$. Note also
  that \eqref{discrete_entropy_inequality} gives the global entropy
  inequality $\sum_{1\le i\le I} m_i\eta(\bsfU_i^{n+1}) \le \sum_{1\le
    i\le I} m_i\eta(\bsfU_i^n)$.
\end{remark}

\begin{remark}
  The meaning of the entropy inequality
  \eqref{entropy_inequality_elemetary_Riemann_pb} might be somewhat
  ambiguous in some cases, especially when $\bu$ is a measure.
  Since it is only the inequality \eqref{entropy_elemetary_Riemann_pb}
  that is really needed in the proof of Theorem~\ref{Thm:entrop_ineq},
  we could replace the assumption
  \eqref{entropy_inequality_elemetary_Riemann_pb} by
  \eqref{entropy_elemetary_Riemann_pb}. This would avoid having to
  invoke measure solutions since $\overline\bu(t,\bn,\bu_L,\bu_R)$
  should always be finite for the Riemann problem
  \eqref{def:Riemann_problem} to have a reasonable (physical) meaning.
\end{remark}

\subsubsection{Cell-based  vs. edge-based viscosity}
In the formulation \eqref{def_of_scheme_dij} the term $\sum_{j\in
  \calI(S_i)} d_{ij} \bsfU_j$ models some edge-based dissipation, \ie
$d_{ij}$ is a dissipation coefficient associated with the pair of
degrees of freedom of indices $(i,j)$. This formulation is related in
spirit to that of local extremum diminishing (LED) schemes developed
for scalar conservation equations in
\cite[Eq. (32)-(33)]{TurekKuzmin2002}, see also
\cite[\S2.1]{Jameson_1995}. It is however a bit difficult to
understand that we are modeling some artificial dissipation by
just staring at \eqref{def_of_scheme_dij}.

We now propose an alternative point of view using a cell-based
viscosity.  The traditional way to introduce dissipation in the finite
element world consists of invoking the weak form of the
Laplacian operator $-\DIV(\nu\GRAD \psi)$. For instance, assuming that
the viscosity field $\nu$ is piecewise constant over each mesh cell
$K\in \calT_h$, we write:
\begin{equation}
\int_\Dom -\DIV(\nu\GRAD \psi) \varphi_i \diff x =
\sum_{K\subset S_i} \nu_K \int_K \GRAD \psi \SCAL \GRAD\varphi_i \diff x. 
\label{weak_laplacian_form}
\end{equation}
Unfortunately, it has been shown in \cite{Guermond_Nazarov_2013} that
the bilinear form
$(\psi,\varphi)\longrightarrow \int_K \GRAD\psi\SCAL\GRAD\varphi\diff
x$
is not robust with respect to the shape of the cells. More
specifically, the convex combination argument, which is essential to
prove the maximum principle for scalar conservation equations in
arbitrary space dimension with continuous finite elements, can be made
to work only if
$\int_{S_{ij}} \GRAD \varphi_i\SCAL \GRAD \varphi_j\diff \bx < 0$ for
all pairs of shape functions, $\varphi_i$, $\varphi_j$, with common
support of nonzero measure. This is the well-known acute angle
condition assumption, which a priori excludes a lot of meshes in
particular in three space dimensions. To avoid this difficulty, it is
proposed in \citep{Guermond_Nazarov_2013} to replace
\eqref{weak_laplacian_form} by
$\sum_{K\subset S_i} \nu_K b_K(\psi, \varphi_i)$, where
\begin{equation}
b_K(\varphi_j,\varphi_i) = 
\begin{cases}
-\vartheta_K|K| & \text{if $i\not= j$},\quad i,j\in \calI(K), \\
|K| & \text{if $i=j$},\quad i,j\in \calI(K),\\
0  &  \text{if  $i\not\in \calI(K)$ or $j\not\in \calI(K)$}. \label{Def_of_bK}
\end{cases} 
\end{equation}
The essential properties of $b_K$ can be
summarized as follows:
\begin{lemma} \label{Lem:prop_of_bK}
  There is $c>0$ depending only on the collection
  $\{(\wK_r,\wP_r,\wSigma_r)\}_{1\le r\le \varpi}$ and the
  shape-regularity, such that the following
  identities hold for all $K\in \calT_h$ and all $u_h, v_h\in P(\calT_h)$:
\begin{align}
b_K(\varphi_i,\varphi_j)&= b_K(\varphi_j,\varphi_i),\qquad
b_K(\varphi_i,\sum_{j\in \calI(K)}\varphi_j) =0.\label{symmetry_and_conservation_of_bK}\\
b_K(u_h,v_h) &= \vartheta_K |K|  \sum_{i\in \calI(K)}  
\sum_{\calI(K)\ni j< i} (\sfU_i -\sfU_j)(\sfV_i-\sfV_j) \label{bK_scalar_product}\\
b_K(u_h,u_h) & \ge c h_K^2 \|\GRAD u_h\|_{\bL^2(K)}^2. \label{bK_coercive}
\end{align}
\end{lemma}%
For instance, when $K$ is a simplex and $\wK$ is the regular simplex,
\ie all the edges are of unit length, it can be shown that
$b_K(\varphi_i,\varphi_i) = \kappa \int_{K} \polJ_K\tr (\GRAD
\varphi_j) \SCAL \polJ_K\tr (\GRAD \varphi_i)\diff x$ and
$b_K(\varphi_j,\varphi_i) = -\frac{\kappa}{1-\nf} \int_{K} \polJ_K\tr
(\GRAD \varphi_j) \SCAL \polJ_K\tr (\GRAD \varphi_i)\diff x$ for
$j\not= i$, with $\kappa=\frac12(1+\frac1d)$. Note also that
$b_K(\varphi_j,\varphi_i) \sim h_K^2 \int_{K} (\GRAD \varphi_j) \SCAL
(\GRAD \varphi_i)\diff x$ if $K$ is a regular simplex, thereby showing
the connection between $b_K$ and the more familiar bilinear form
associated with the Laplacian. 
One key argument from \cite{Guermond_Nazarov_2013} is the
recognition that the bilinear form defined in \eqref{Def_of_bK} has
all the good characteristics of the Laplacian-based diffusion (see
Lemma~\ref{Lem:prop_of_bK}) and makes
the convex combination argument to work independently of the space
dimension and the shape-regularity of the mesh family. 

Hence, instead of \eqref{def_of_scheme_dij}, we could also  compute
$\bu_h^{n+1}$ by
\begin{equation}
m_i \frac{\bsfU_i^{n+1}-\bsfU_i^n}{\dt} 
+ \int_\Dom  \DIV(\Pi_h\bef(\bu_h^n)) \varphi_i \diff x 
 +  \sum_{K\in \calT_h} \nu_K^n \sum_{j\in \calI(K)} \bsfU^n_j b_K(\varphi_j,\varphi_i)
 = 0,  \label{def_of_scheme}
\end{equation}
where $\{\nu_K^n\}_{K\in\calT_h}$ is a piecewise constant artificial
viscosity scalar field. 
\begin{theorem} \label{Thm:UL_is_invariant_cell_based}
Let  $\{\nu_K^n\}_{K\in\calT_h}$ be defined by
\begin{equation}
  \nu_K := \max_{i\ne j\in \calI(K)} 
  \frac{\lambda_{\max}(\bn_{ij},\bsfU_i,\bsfU_j) \|\bc_{ij}\|_{\ell^2}}{\sum_{T\subset S_{ij}} 
    -b_T(\varphi_j,\varphi_i)}. \label{Def_of_nuh_graph} 
\end{equation}
Then the conclusions of Theorem~\ref{Thm:UL_is_invariant} and
Theorem~\ref{Thm:entrop_ineq} hold under the assumptions
\eqref{lambda_max_bounded} and \eqref{CFL_assumption} and with the
solution process $\bu_h^n\longrightarrow \bu_h^{n+1}$, $n\ge 0$,
defined by \eqref{def_of_scheme}.
\end{theorem}
\begin{proof}
Let us denote $\tilde d_{ij} := -\sum_{K\in S_{ij}} \nu_K^n b_K(\varphi_j,\varphi_i)$, then
\eqref{def_of_scheme} can be recast as follows:
\[
m_i \frac{\bsfU_i^{n+1}-\bsfU_i^n}{\dt} 
+ \sum_{j\in \calI(S_i)} \bef(\bsfU^{n}_j) \SCAL \bc_{ij} 
 -  \bsfU^n_j \tilde d_{ij}
 = 0,  
\]
which in turn implies that
\[
  \bsfU_i^{n+1}  = \bsfU_i^n
  \Big(1 - \sum_{i\ne j\in \calI(S_i)} \frac{2\dt \tilde d_{ij}}{m_i} \Big) 
  + \sum_{i\ne j\in \calI(S_i)}\frac{2\dt \tilde d_{ij}}{m_i}\overline\bsfU_{ij}^{n+1}.   
\]
where we have introduced the auxiliary quantities
\[
\overline\bsfU_{ij}^{n+1} := \frac12(\bsfU^n_j+\bsfU^n_i) -
(\bn_{ij}\SCAL\bef(\bsfU_j^n) - \bn_{ij}\SCAL \bef(\bsfU_i^n)) 
\frac{\|\bc_{ij}\|_{\ell^2}}{2\tilde
  d_{ij}}.
\]
Here again $\overline\bsfU_{ij}^{n+1}$ is of the form
$\overline\bu(t,\bn_{ij},\bsfU_i,\bsfU_j)$ as defined in
\eqref{elemetary_Riemann_pb} with the fake time
$t=\|\bc_{ij}\|_{\ell^2}/2\tilde d_{ij}$, hence we need to make sure
that $\lambda_{\max}(\bn_{ij},\bu_L,\bu_R)
\|\bc_{ij}\|_{\ell^2}/2\tilde d_{ij} \le \frac12$ to preserve the
invariant domain property. Recalling that $d_{ij}$ has been defined by
$d_{ij}:=\lambda_{\max}(\bn_{ij},\bu_L,\bu_R) \|\bc_{ij}\|_{\ell^2}$
(see \eqref{introduction_dij}), the above condition reduces to showing that
$d_{ij}\le \tilde d_{ij}$. The definitions of $\nu_K$ and $\tilde d_{ij}$ implies that
\begin{align*}
\tilde d_{ij} & = -\sum_{K\in S_{ij}} \nu_K^n b_K(\varphi_j,\varphi_i)
\ge -\sum_{K\in S_{ij}}  \frac{\lambda_{\max}(\bn_{ij},\bsfU_i,\bsfU_j) \|\bc_{ij}\|_{\ell^2}}{\sum_{T\subset S_{ij}} 
    -b_T(\varphi_j,\varphi_i)} b_K(\varphi_j,\varphi_i)\\
&\ge -\sum_{K\in S_{ij}}  \frac{d_{ij}}{\sum_{T\subset S_{ij}} 
    -b_T(\varphi_j,\varphi_i)} b_K(\varphi_j,\varphi_i) = d_{ij},
\end{align*}
whence the desired result. We now prove that $1 - \sum_{i\ne j\in
  \calI(S_i)} \frac{2\dt \tilde d_{ij}}{m_i}\ge 0$ under the CFL
condition \eqref{CFL_assumption}. From the proof of
Theorem~\ref{Thm:UL_is_invariant} we have $d_{ij} \le \lambda_{\max}(A)
\hmin^{-1}\mu_{\max}|S_{ij}|$, hence
\begin{align*}
\nu_K & \le  \frac{\lambda_{\max}(A)\mu_{\max}}{\hmin} \max_{k\ne l\in I(K)}
\frac{|S_{kl}|}{\sum_{T\subset S_{kl}}-b_T(\varphi_k,\varphi_l)},
\end{align*}
which in turn implies that
\begin{align*}
\tilde d_{ij} & \le \frac{\lambda_{\max}(A)\mu_{\max}}{\hmin} 
\sum_{K\subset S_{ij}} -b_K(\varphi_i,\varphi_j) \max_{k\ne l\in I(K)}
\frac{|S_{kl}|}{\sum_{T\subset S_{kl}}-b_T(\varphi_k,\varphi_l)}
\end{align*}
Recalling the definition of $b_T((\varphi_k,\varphi_l)$ we have
$\sum_{T\subset S_{kl}}-b_T(\varphi_k,\varphi_l)\ge \vartheta_{\min}
|T| = \vartheta_{\min} |S_{kl}|$; hence
\begin{align*}
\tilde d_{ij} & \le \frac{\lambda_{\max}(A)\mu_{\max}}{\vartheta_{\min}\hmin} 
\sum_{K\subset S_{ij}} -b_K(\varphi_i,\varphi_j) =\frac{\lambda_{\max}(A)\mu_{\max}}{\vartheta_{\min}\hmin} 
\sum_{K\subset S_{ij}} \vartheta_{K}|K|.
\end{align*}
Finally we have
\begin{align*}
-\tilde d_{ii}:=\sum_{i\ne j\in \calI(S_i)}\tilde d_{ij} & \le \frac{\lambda_{\max}(A)\mu_{\max}}{\vartheta_{\min}\hmin} 
\sum_{i\ne j\in \calI(S_i)} \sum_{K\subset S_{ij}} \vartheta_{K}|K| = \frac{\lambda_{\max}(A)\mu_{\max}}{\vartheta_{\min}\hmin}  |S_{i}|.
\end{align*}
This means that the bound on $-\tilde d_{ii}$ is the same as that on
$-d_{ii}$ in the proof of Theorem~\ref{Thm:UL_is_invariant}. This
concludes the proof.
\end{proof}

\section{Numerical illustrations}
\label{Sec:numerical}
We illustrate in this section the method described in the paper, \ie
\eqref{def_of_scheme_dij}-\eqref{Def_of_dij}, and discuss possible
variants.

\subsection{Invariant domain property and convergence issues}
We give in this section a counter-example showing that a method that
is formally first-order consistent and satisfies the invariant domain
property may not necessarily be convergent.

To illustrate or point, let us focus our attention on scalar
conservation equations and let us consider an algebraic approach that
is sometimes used in the literature, see \eg
\cite[p.~163]{KuzminLoehnerTurek2004},
\cite[Eq. (32)-(33)]{TurekKuzmin2002}. Instead of constructing a
convex combination involving (entropy satisfying) intermediate states
like in \eqref{Unplusone_convex_combination_dij}, we re-write
\eqref{Unplusone_convex_combination_dij_rewritten} as follows:
 \begin{equation}
  m_i \frac{\sfU_i^{n+1}-\sfU_i^n}{\dt} 
  = -\sum_{i\ne j\in \calI(S_i)} (\bef(\sfU_j^n) - \bef(\sfU_i^n)) \SCAL \bc_{ij}
  + \sum_{j\in \calI(S_i)} d_{ij}\sfU^n_j.
\end{equation}
Or, equivalently
\begin{equation}
m_i \frac{\sfU_i^{n+1}-\sfU_i^n}{\dt} 
  = -\sum_{i\ne j\in \calI(S_i)} \frac{\bef(\sfU_j^n) -\bef(\sfU_i^n)}{\sfU_j^n - \sfU_i^n} \SCAL \bc_{ij} (\sfU_j^n -\sfU_i^n)
  + \sum_{j\in \calI(S_i)} d_{ij}\sfU^n_j. \label{scheme_with_kuzmin_flux}
\end{equation}
Let us set
$k_{ij}:=\frac{\bef(\sfU_j^n) -\bef(\sfU_i^n)}{\sfU_j^n - \sfU_i^n}
\SCAL \bc_{ij}$, (with $k_{ij}:=0$ if $\sfU_j^n=\sfU_i^n$), then
\begin{equation}
\sfU_i^{n+1}
  = \sfU_i^n\Big( 1 - \frac{\dt}{m_i}\sum_{i\ne j\in \calI(S_i)}
  (-k_{ij} + d_{ij})\Big)+
\sum_{i\ne j\in \calI(S_i)} \frac{\dt}{m_i}(-k_{ij} + d_{ij})\sfU_j^n. \label{kuzmin_scheme}
\end{equation}
Let us finally set 
\begin{equation}
d_{ij}:=\max(0,k_{ij},k_{ji}), \ i\ne j, \quad \text{and}\quad
d_{ii}:= -\sum_{i\ne j\in \calI(S_i)} d_{ij}. \label{kuzmin_dij}
\end{equation}
This choice implies that $-k_{ij} + d_{ij}\ge 0$ for all 
$i\in \intset{1}{N}$, $j\in \calI(S_i)$. As a result,
$\sfU_i^{n+1} \in \text{conv}\{\sfU_j^n,\ j\in \calI(S_i)\}$ under the
appropriate CFL condition; hence, the solution process
$u_h^n \longmapsto u_h^{n+1}$ described above in
\eqref{kuzmin_scheme}-\eqref{kuzmin_dij} satisfies the maximum
principle. Although, this technique looks reasonable a priori, it
turns out that it is not diffusive enough to handle general fluxes as
discussed in \cite[\S3.3]{Guermond_Popov_2015}. The convergence result
established in \citep{Guermond_Popov_2015} requires an estimation of
the wave speed that is more accurate than just the average speed
$\bn_{ij}\SCAL \frac{\bef(\sfU_j^n) -\bef(\sfU_i^n)}{\sfU_j^n - \sfU_i^n}$,
which is invoked in the above definition. This definition of the wave
speed is correct in shocks, \ie if the Riemann problem with data
$(\sfU_i,\sfU_j)$ is a simple shock; but it may not be sufficient if the
Riemann solution is an expansion or a composite wave, which is
likely to be the case if $\bef$ is not convex.

We now illustrate numerically the observation made above.
We consider the so-called KPP problem proposed in \cite{KPP_2007}.
It is a two-dimensional scalar conservation equation with a non-convex
flux:
\begin{equation}
  \partial_t u + \nabla \cdot \bef(u) = 0, \quad
  u(\bx,0) = u_0(\bx)= \left\{
    \begin{aligned}
      &\tfrac{14\pi}{4}, \quad \mbox{if } \sqrt{x^2+y^2} \leq 1, \\
      &\tfrac{\pi}{4}, \quad \mbox{otherwise}.
    \end{aligned}
  \right.,
\end{equation}
where $\bef(u) = (\sin u, \cos u)$. This is a challenging test case
for many high-order numerical schemes because the solution has a
two-dimensional composite wave structure. For example, it has been
shown in \citep{KPP_2007} that some central-upwind schemes based on
WENO5, Minmod 2 and SuperBee reconstructions converge to non-entropic
solutions.

The computational domain $[-2,2]\CROSS[-2.5,1.5]$ is triangulated
using non-uniform meshes and the solution is approximated up to $t=1$
using continuous $\polP_1$ finite elements (29871 nodes, 59100
triangles). The time stepping is done with SSP RK3. The solution shown
in the left panel of Figure~\ref{fig:KPP} is obtained using
\eqref{kuzmin_dij} for the definition of $d_{ij}$.  The numerical
solution produces  very sharp, non-oscillating, entropy violating shocks, the reason
being that the artificial viscosity is not large enough.  Note that
the solution is maximum principle satisfying (the local maximum
principle is satisfied at every grid point and every time step) and no
spurious oscillations are visible. The numerical process converges to
a nice-looking (wrong) piecewise smooth weak solution.  The numerical
solution shown in the right panel of Figure~\ref{fig:KPP} is obtained
by using our definition of $d_{ij}$, \eqref{Def_of_dij} (note in passing that the results
obtained with \eqref{def_of_scheme}-\eqref{Def_of_nuh_graph} together
with \eqref{Def_of_dij} are indistinguishable from this solution).  The
expected helicoidal composite wave is clearly visible; this is the
unique entropy satisfying solution.
\begin{figure}[H]
\centering{
 \includegraphics[width=0.40\textwidth]{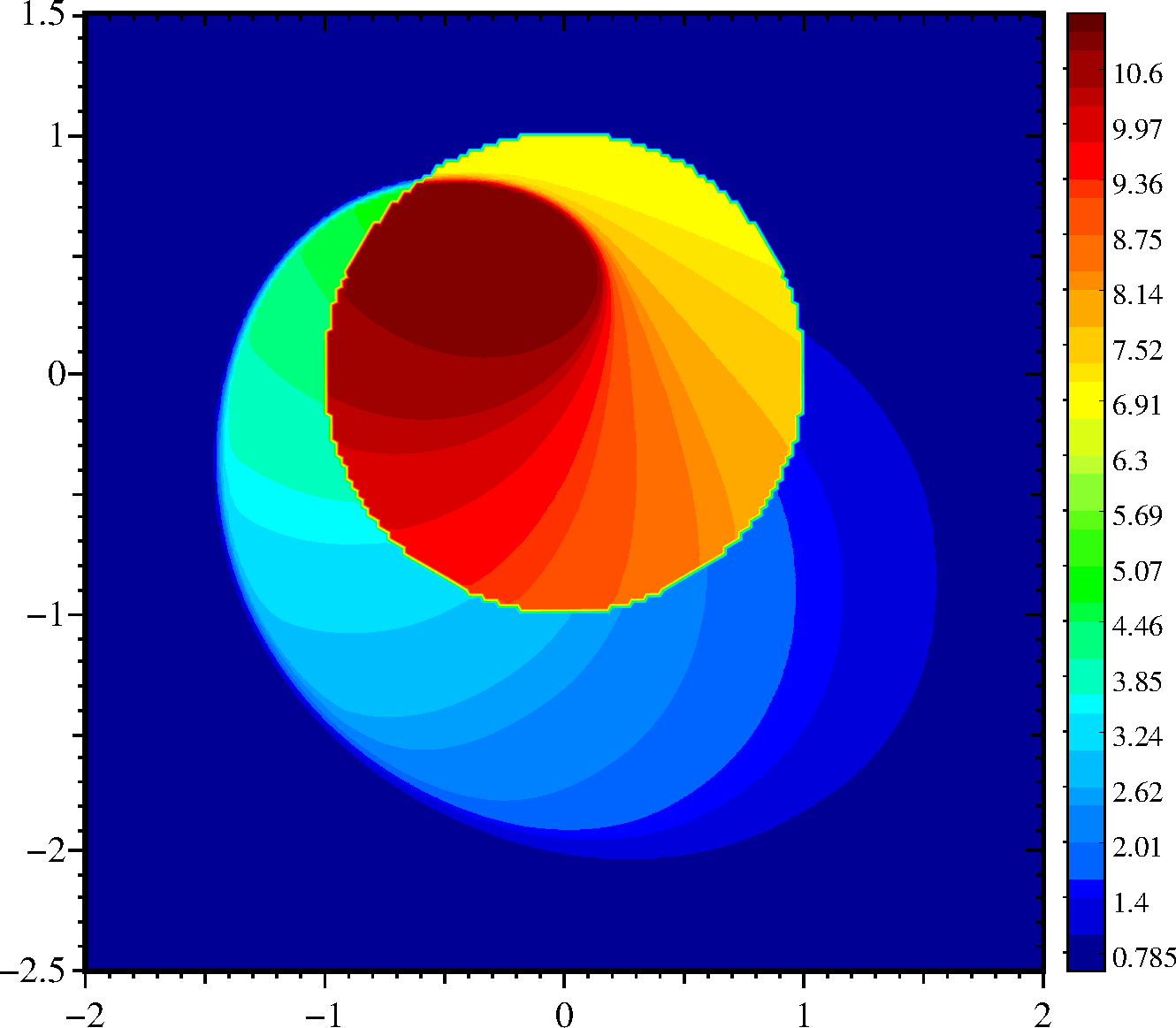}\hfil
 \includegraphics[width=0.40\textwidth]{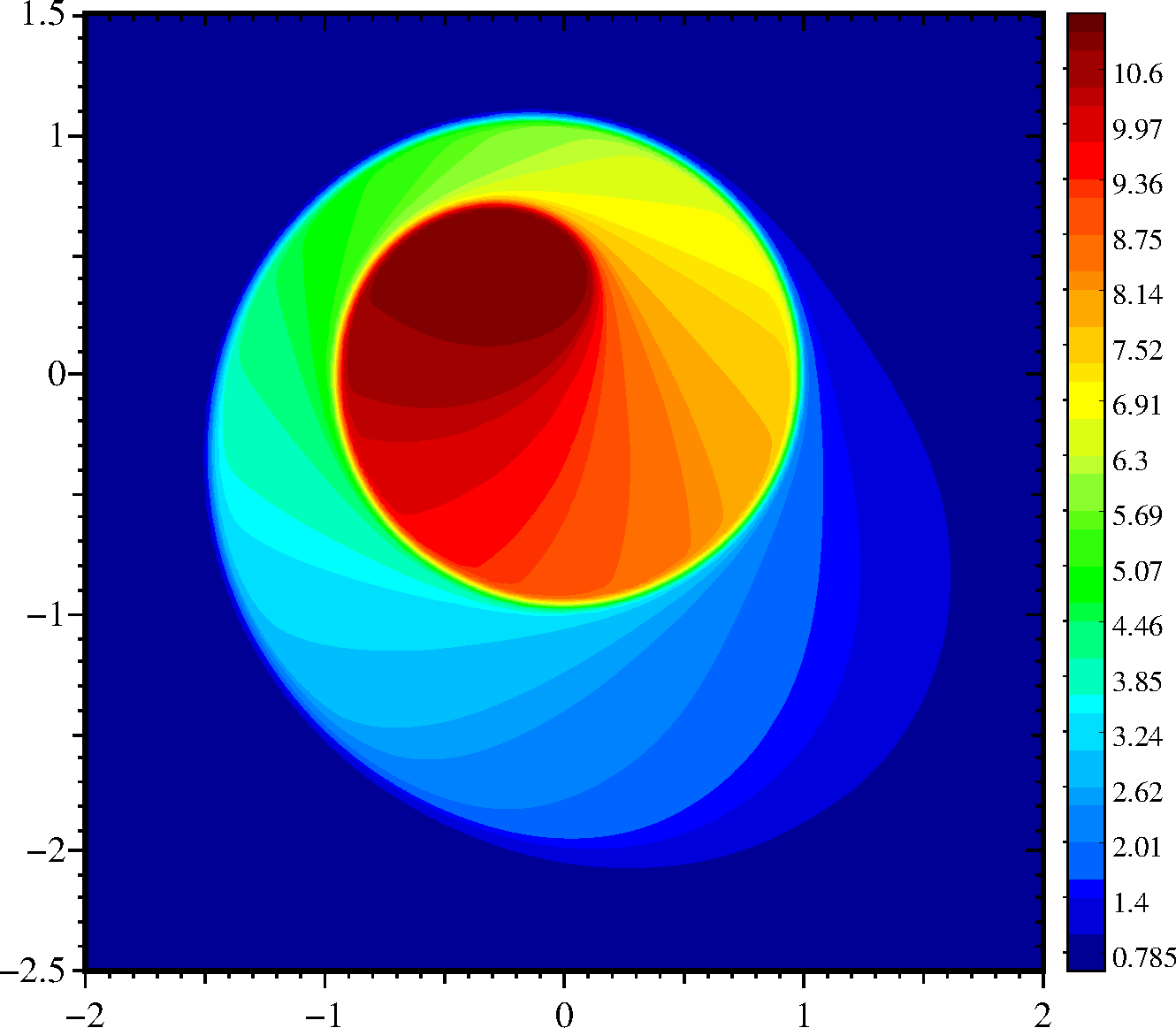}
}
\caption{KPP solution with continuous $\polP_1$ elements (29871 nodes,
  59100 triangles). Left: entropy violating solution using
  \eqref{def_of_scheme_dij}-\eqref{kuzmin_dij}; Right: entropy
  satisfying solution using
  \eqref{def_of_scheme_dij}-\eqref{Def_of_dij}.}\label{fig:KPP}
\end{figure}

In conclusion, the above counter-example shows that satisfying the invariant
domain property/maximum principle does not imply convergence, even for
a first-order method. It is also essential that the method 
satisfies local entropy inequalities to be convergent; this is the case of our method
\eqref{def_of_scheme_dij}-\eqref{Def_of_dij} (see
Theorem~\ref{Thm:entrop_ineq}), but it is not the case of the algebraic
method \eqref{kuzmin_scheme}-\eqref{kuzmin_dij}.

\begin{remark} The reader should be aware that we are citing
  \cite[p.~163]{KuzminLoehnerTurek2004},
  \cite[Eq. (32)-(33)]{TurekKuzmin2002} a little bit out of context.
  The scheme as originally presented in the above references was only
  meant to solve the linear transport equation, and as such it is a
  perfectly good method. Problems  arise with \eqref{kuzmin_dij} only
  when one extends the methodology to nonlinear nonconvex fluxes, as
  we did in \eqref{scheme_with_kuzmin_flux}.
\end{remark}

\subsection{Special meshes}
The construction of the intermediate states in
\eqref{Unplusone_convex_combination_dij_rewritten} is not unique.  For
instance we can extend a construction used by 
\cite[Cor. 1]{Hoff_1979}
in one space dimension for the $p$-system.
Let us assume that $i\in \{1,\ldots, N\}$ is such that every
$j\in \calI(S_i)\setminus\{i\}$, there is a unique $\sigma_i(j)\in
\calI(S_i)\setminus\{i,j\}$ such that $\bc_{ij}:=\int_{S_i} \phi_i
\GRAD \phi_j \diff x = -\int_{S_i} \phi_i \GRAD \phi_{\sigma_i(j)}
\diff x =: -\bc_{i\sigma_i(j)}$. This property holds in one space
dimension for any mesh if $\ba_i$ is an interior node. It holds in
higher space dimension provided the mesh has symmetry properties and
$\ba_i$ is an interior node; for instance it holds if the mesh is
centrosymmetric, \ie the support of $\phi_i$ is symmetric with respect
to the node $\ba_i$ for any $i\in\{1,\ldots,N\}$. Then we can re-write
\eqref{update_cijdij_dij} as follows:
\begin{equation}
  m_i \frac{\bsfU_i^{n+1}-\bsfU_i^n}{\dt} 
  = d_{ii}\bsfU^n_i - \sum_{j\in \calJ(S_i)} (\bef(\bsfU_j^n) - \bef(\bsfU_{\sigma_i(j)}^n)) \SCAL \bc_{ij}
  + d_{ij}\bsfU^n_j+ d_{i\sigma_i(j)}\bsfU^n_{\sigma_i(j)}.  
\end{equation}
where the set $\calJ(S_i)\subset \calI(S_i)$ is such that $\sigma_i : \calJ(S_i)
\longrightarrow \sigma_i(\calJ(S_i))$ is bijective and 
$\calJ(S_i) \cup \sigma_i(\calJ(S_i)) = \calI(S_i)\setminus\{i\}$. Then 
upon recalling that $d_{ii} := -\sum_{j\in \calJ(S_i)}(d_{ij}+d_{i\sigma_i(j)})$,
we have
\begin{equation}
\bsfU_i^{n+1} = \bsfU_i^n\bigg(1 
-\sum_{j\in \calJ(S_i)}\frac{\dt}{m_i}(d_{ij}+d_{i\sigma_i(j)}) \bigg) 
+ \sum_{j\in \calJ(S_i)} \frac{\dt (d_{ij}+d_{i\sigma_i(j)})}{m_i}
\overline\bsfU_{ij}^{n+1},
\end{equation}
where we have defined  the intermediate state $\overline\bsfU_{ij}^{n+1}$
by
\begin{equation}
  \overline\bsfU_{ij}^{n+1} = 
 \frac{d_{i\sigma_i(j)}}{d_{ij}+d_{i\sigma_i(j)}}\bsfU_{\sigma_i(j)}^n
+\frac{d_{ij}}{d_{ij}+d_{i\sigma_i(j)}} \bsfU_j^n 
    - (\bef(\bsfU_j^n) - \bef(\bsfU_{\sigma_i(j)}^n)) 
    \SCAL \frac{\bc_{ij}}{d_{ij}+d_{i\sigma_i(j)}}.
\end{equation}
The state $\overline\bsfU_{ij}^{n+1}$ is of the form
$\overline\bu(t,\bn_{ij},\bsfU_{\sigma_i(j)}^{n},\bsfU_{j}^{n}) :=
\int_{\alpha_L}^{\alpha_R}
\bu(\bn_{ij},\bsfU_{\sigma_i(j)}^{n},\bsfU_{j}^{n})(x,t) \diff x$,
where $\alpha_L= -\frac{d_{i\sigma_i(j)}}{d_{ij}+d_{i\sigma_i(j)}}$,
$\alpha_R= \frac{d_{ij}}{d_{ij}+d_{i\sigma_i(j)}}$ and $t:=
\frac{\|\bc_{ij}\|_{\ell^2}}{d_{ij}+d_{i\sigma_i(j)}}$, provided
\begin{align}
  d_{i\sigma_i(j)} &\ge
  (\lambda_1^-)^-(\bn_{ij},\bsfU_{\sigma_i(j)}^n, \bsfU_j^n)
  \| \bc_{ij}\|_{\ell^2}, && \forall j \in \calJ(S_i), \\
  d_{ij} &\ge (\lambda_m^+)^+(\bn_{ij},\bsfU_{\sigma_i(j)}^n,
  \bsfU_j^n) \| \bc_{ij}\|_{\ell^2}, && \forall j \in \calJ(S_i),
\end{align}
where we defined $x^+ = \max(x,0)$ and $x^-=-\min(x,0)$.  A sufficient
condition that implies both the above inequalities and is independent
of the choice of the set $\calJ_i(S_i)$ is
\begin{equation}
\min(d_{ij},d_{i\sigma_i(j)}) \ge 
\lambda_{\max}(\bn_{ij},\bsfU_{\sigma_i(j)}^n, \bsfU_j^n)\| \bc_{ij}\|_{\ell^2},
\qquad j \in \calJ(S_i).
\end{equation}

Note that the above argument holds only if $\ba_i$ is an interior node
satisfying the symmetry property $\bc_{ij}=-\bc_{i\sigma_i(j)}$.  If
this is not the case, then we can always use the lower bound
\eqref{lower_bound_on_dij}, \ie $d_{ij} \ge
\lambda_{\max}(\bn_{ij},\bsfU_{i}^n,\bsfU_{j}^n)
\|\bc_{ij}\|_{\ell^2}$.

In conclusion the diffusion matrix $(d_{ij})_{1\le i,j\le N}$ can be
constructed as follows: (1) For every node $i$ satisfying the symmetry
property $\bc_{ij}=-\bc_{i\sigma_i(j)}$ for every $j\in \calJ(S_i)$,
we define $\widetilde{d}_{ij}=\widetilde{d}_{i\sigma_i(j)}
=\lambda_{\max}(\bn_{ij},\bsfU_{\sigma_i(j)}^n, \bsfU_j^n)\|
\bc_{ij}\|_{\ell^2}$; (2) For every other index $i$ not satisfying the
symmetry property mentioned above, we define $\widetilde{d}_{ij} =
\lambda_{\max}(\bn_{ij},\bsfU_{i}^n,\bsfU_{j}^n)
\|\bc_{ij}\|_{\ell^2}$; (3) We construct the diffusion matrix by
setting $d_{ij} := \max(\widetilde{d}_{ij},\widetilde{d}_{ji})$ for
$j\ne i$ and $d_{ii} := -\sum_{i\ne j\in \calI(S_i)} d_{ij}$.  This
construction guarantees conservation, \ie $\sum_{i\in \calI(S_j)}
d_{ij} =0$ and first-order consistency, \ie $\sum_{j\in \calI(S_i)}
d_{ij} =0$.

\begin{remark}
  Quite surprisingly, in the case of scalar linear transport the above
  construction and the construction done in
  \S\ref{Sec:convex_argument}, (see definition \eqref{Def_of_dij})
  give the same scheme (\ie the same CFL).
\end{remark}

\subsection{Invariant domain property vs. monotonicity}
We show in this section that the invariance property and what is usually
understood in the literature as monotonicity are two different
concepts and just looking at monotonicity may be misleading.
\subsection{p-system}
We consider the p-system and solve the Riemann problem corresponding
to the initial data $(v_L,u_L)=(1,0)$, $(v_R,u_R)=
(2^{\frac{2}{\gamma-1}},\tfrac{1}{\gamma-1})$.  The computational
domain is the segment $[0,1]$ and the separation between the left and
right states is set at $x_0=0.75$.  The solution is a single rarefaction
wave from the first family (\ie $w_1(v_L,u_L) = w_1(v_R,u_R)$):
\begin{equation}
v(x,t) = \begin{cases}
1 & \text{if $\frac{x-x_0}{t} \le -1$} \\
(\frac{x_0-x}{t})^{\frac{-2}{\gamma+1}}
 &\text{if $-1 \le \frac{x-x_0}{t} \le -2^{-\frac{\gamma+1}{\gamma-1}}$} \\
 2^{\frac{2}{\gamma-1}} & \text{otherwise}
\end{cases}
\end{equation}
\begin{equation}
u(x,t) = \begin{cases}
0 & \text{if $\frac{x-x_0}{t} \le -1$} \\
\frac{2}{\gamma-1}\left(1-(\frac{x_0-x}{t})^{\frac{\gamma-1}{\gamma+1}}\right) 
 &\text{if $-1 \le \frac{x-x_0}{t} \le -2^{-\frac{\gamma+1}{\gamma-1}}$} \\
\frac{1}{\gamma-1} & \text{otherwise}
\end{cases}
\end{equation}
This case is such that $(v^*,u^*)=(v_R,u_R)$, hence the second wave
corresponding to the eigenvalues $\lambda_2^{\pm}$ is
not present.  We use continuous piecewise linear finite elements with
the algorithm~\eqref{def_of_scheme_dij}-\eqref{Def_of_dij}. The time
stepping is done with the SSP RK3 technique. We show the profile of
$v$ at $t=0.75$ in Figure~\ref{Fig:p-system} for meshes composed of
$10^3,2\CROSS 10^3, 4\CROSS 10^3, 10^4, 2\CROSS 10^4, 4\CROSS 10^4$,
$10^5$, $2\CROSS 10^5$ cells. 
\begin{figure}[ht]
\centerline{\includegraphics[width=0.33\textwidth]{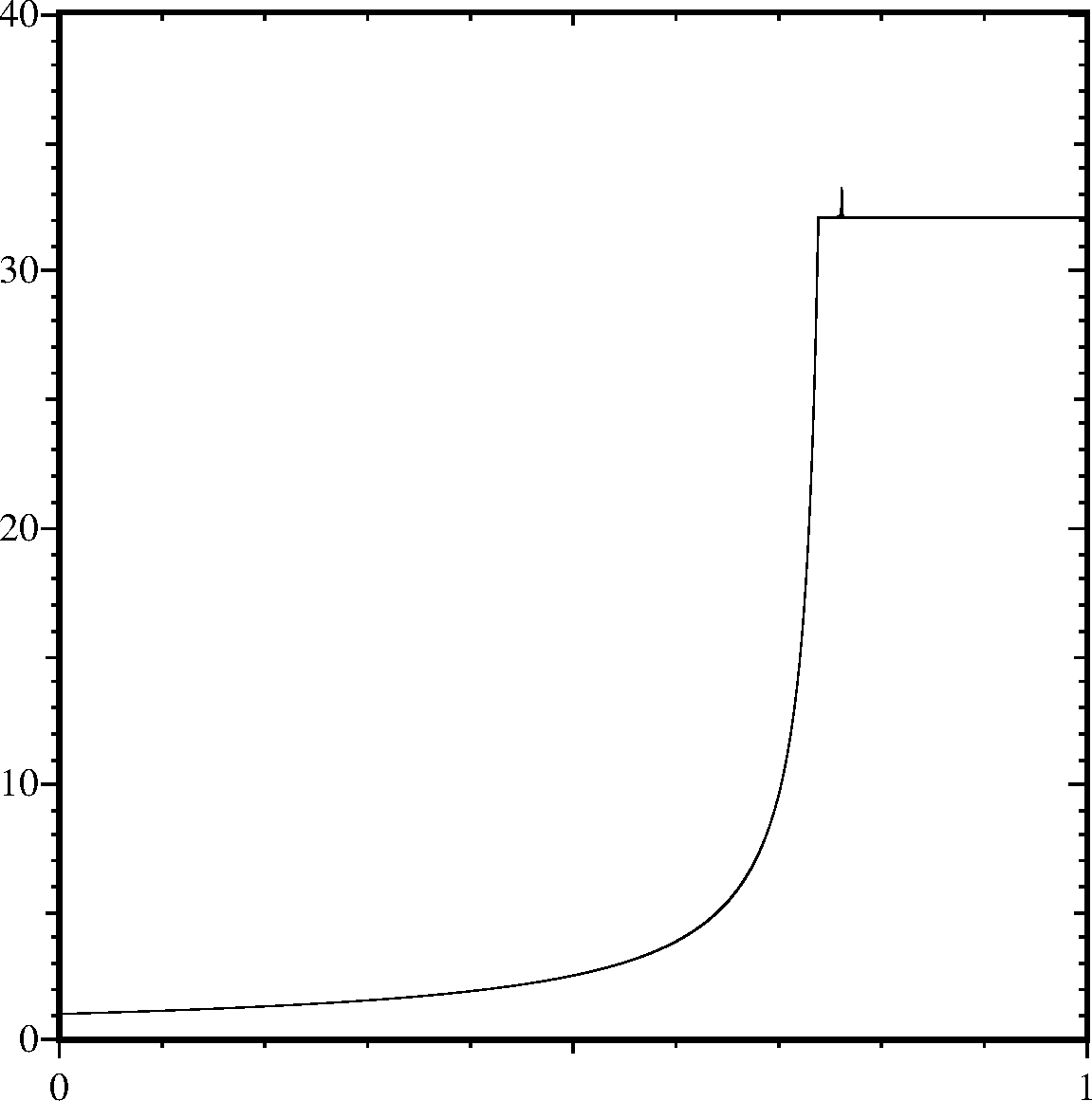}\hfil
\includegraphics[width=0.33\textwidth]{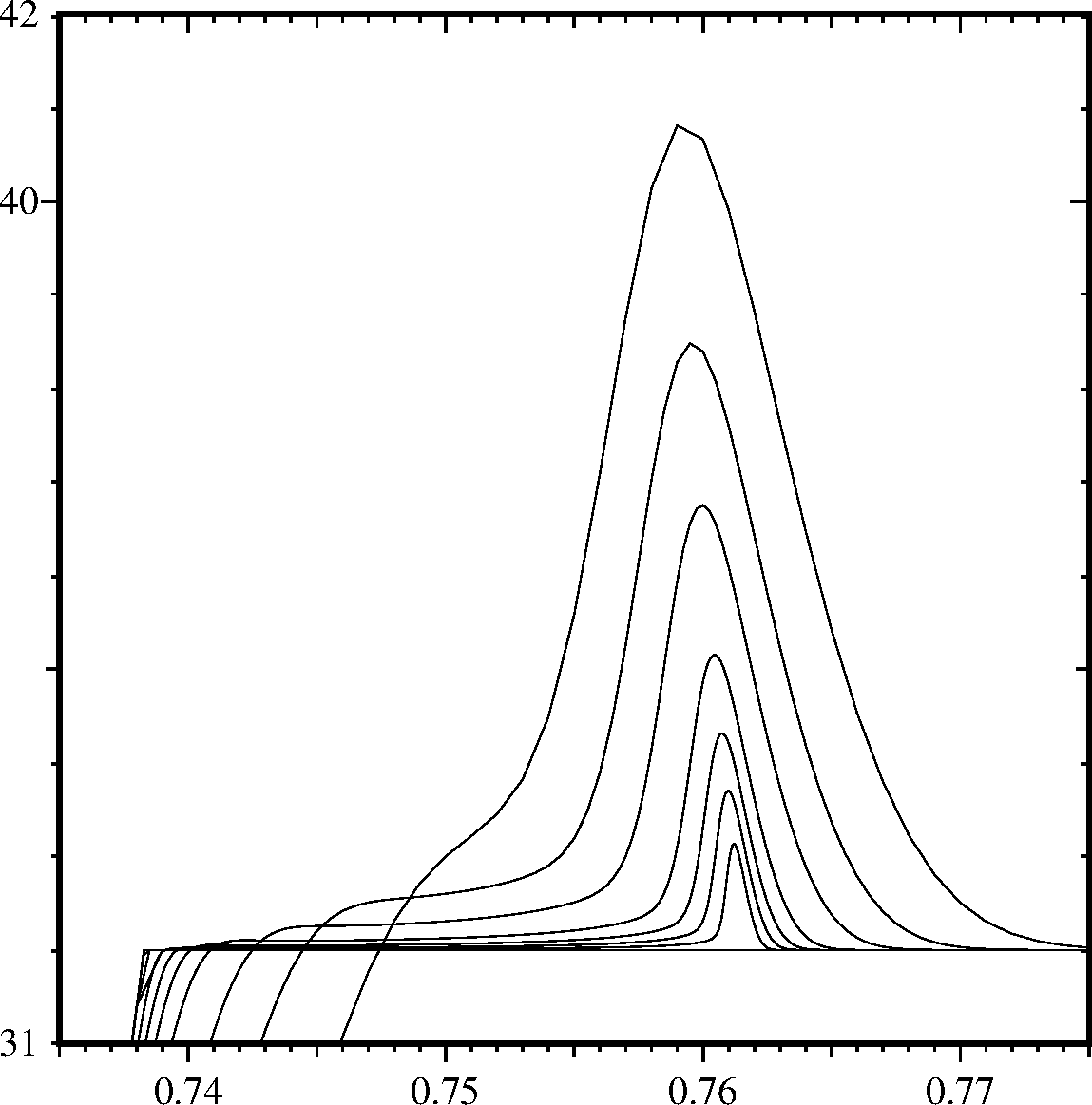}
}
\caption{Left: $v$-profile for the p-system at $t=0.75$, $10^5$ grid points.
  Right: close up view of the $v$-profile for various grid sizes:
$10^3,2\CROSS 10^3, 4\CROSS 10^3, 10^4, 2\CROSS 10^4, 4\CROSS 10^4$,
$10^5$ grid points.}
\label{Fig:p-system}
\end{figure}
We observe that the profile is not monotone. There is an overshoot at
the right of the foot of the (left-going) wave. Actually this
overshoot does not violate the invariant domain property; we have
verified numerically that, at every time step and for every grid point
in each mesh, the numerical solution is in the smallest invariant
domain of type \eqref{speed_upper_bound_p_system} that contains the
piecewise linear approximation of the initial data. This result seems
a bit surprising, but it is perfectly compatible with
Theorem~\ref{Thm:UL_is_invariant}. Since the numerical solution cannot
stay on the exact rarefaction wave (green line connecting $\bsfU_L$
and $\bsfU_L$ in Figure~\ref{fig:Overshoot}), the second wave
reappears in the form of an overshoot at the end of the rarefaction
wave (see right panel of the Figure~\ref{Fig:p-system}).
\begin{figure}[ht]
\centering{
 \includegraphics[width=0.55\textwidth]{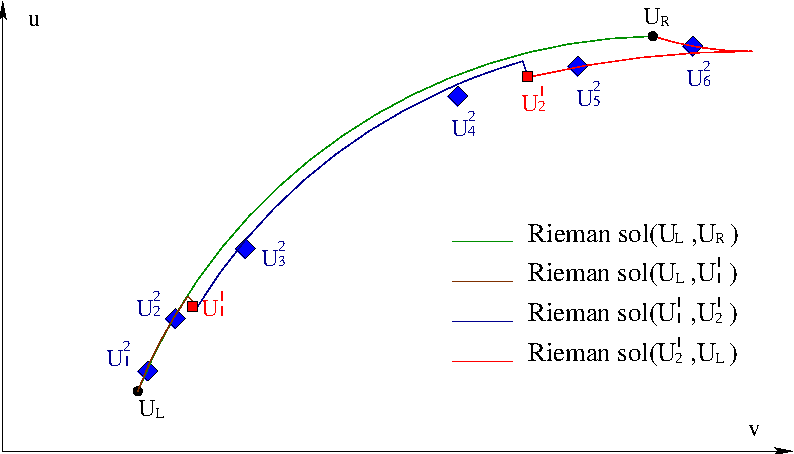}
}
\caption{The overshooting mechanism for a single rarefaction wave in
  the phase space for the p-system. Initial data in black; additional points after one
  time step in red; after two time steps in blue. Observe the position
  of $\bsfU^2_6$.}
\label{fig:Overshoot}
\end{figure}

Let $(\bsfU_L\dots,\bsfU_L,\bsfU_R\dots,\bsfU_R)$ be the initial
sequence of degrees of freedom. After one time step two additional
points appear in the phase space, denoted on
Figure~\ref{fig:Overshoot} by $\bsfU_1^1$ and $\bsfU_2^1$. Because of
the invariant domain property, these points are under the rarefaction
wave. Then the sequence of degrees of freedom at time $t=\dt$ is
$(\bsfU_L\dots,\bsfU_L,\bsfU^1_1,\bsfU^2_2, \bsfU_R\dots, \bsfU_R)$.
Six additional points $\bsfU^2_1, \ldots, \bsfU^2_6$ appear after two
time steps and the sequence of degrees of freedom at time $t=2\dt$ is
$(\bsfU_L\dots,\bsfU_L,\bsfU^2_1, \ldots, \bsfU^2_6,
\bsfU_R\dots,\bsfU_R)$.  The point $\bsfU^2_6$ is the one whose
$v$-component may overshoot because the exact solution of the Riemann
problem with the left state $\bsfU^1_2$ and the right state $\bsfU_R$
is composed of two rarefaction waves and the maximum value of $v$ on
these rarefactions is necessarily larger than $v_R$ (see red line in
Figure~\ref{fig:Overshoot}).  Note that this is not a Gibbs phenomenon
at all; in particular the amplitude of the overshoot decreases as the
mesh is refined as shown in the close up view in the right panel of
the Figure~\ref{Fig:p-system}. This phenomenon is actually very common
in numerical simulations of hyperbolic systems but is rarely
discussed; it is sometimes called "start up error" in the literature,
see for example the comments on page 592 in
\cite{Kurganov_Tadmor_2002} and the comments at the bottom of page
1005 in \cite{Liska_Wendroff}.
The (relative) $L^1$-norm of the error on both $v$ and $u$ at $t=0.75$
is shown in Table~\ref{Table:converge_psystem}.  The method converges
with an order close to $0.9$.
\begin{table}[ht]
\centerline{\begin{tabular}{||c|c|c|c|c||} \hline
$1/h$           & $v$        & rate  & $u$        & rate\\ \hline
$10^3$          & 1.8632(-2) & -     & 7.2261(-3) & \\ \hline
$2\CROSS 10^3$  & 1.0350(-2) & 0.85  & 3.9239(-3) & 0.88\\ \hline
$4\CROSS 10^3$  & 5.6769(-3) & 0.87  & 2.1173(-3) & 0.89\\ \hline
$10^4$          & 2.5318(-3) & 0.88  & 9.2888(-4) & 0.90\\ \hline
$2\CROSS 10^4$  & 1.3644(-3) & 0.89  & 4.9541(-4) & 0.91\\ \hline
$4\CROSS 10^4$  & 7.3151(-4) & 0.90  & 2.6319(-4) & 0.91\\ \hline
$1\CROSS 10^5$  & 2.9695(-4) & 0.98  & 1.1352(-4) & 0.92\\ \hline
$2\CROSS 10^5$  & 1.5838(-4) & 0.91  & 5.9869(-5) & 0.92\\ \hline
\end{tabular}}
\caption{Convergence rates for the p-system}
\label{Table:converge_psystem}
\end{table}

\subsection{Euler in 1D (Leblanc shocktube)}
We consider now the compressible Euler equations. We solve the Riemann
problem also known in the literature as the Leblanc Shocktube. The data
are as follows: $\gamma=\frac53$ and
\begin{align*}
\rho_L & = 1.000,\quad u_L=0.0,\quad p_L = 0.1\\
\rho_R & = 0.001,\quad u_R=0.0,\quad p_R= 10^{-15}.
\end{align*}
\begin{figure}[h]
\centerline{\includegraphics[width=0.54\textwidth]{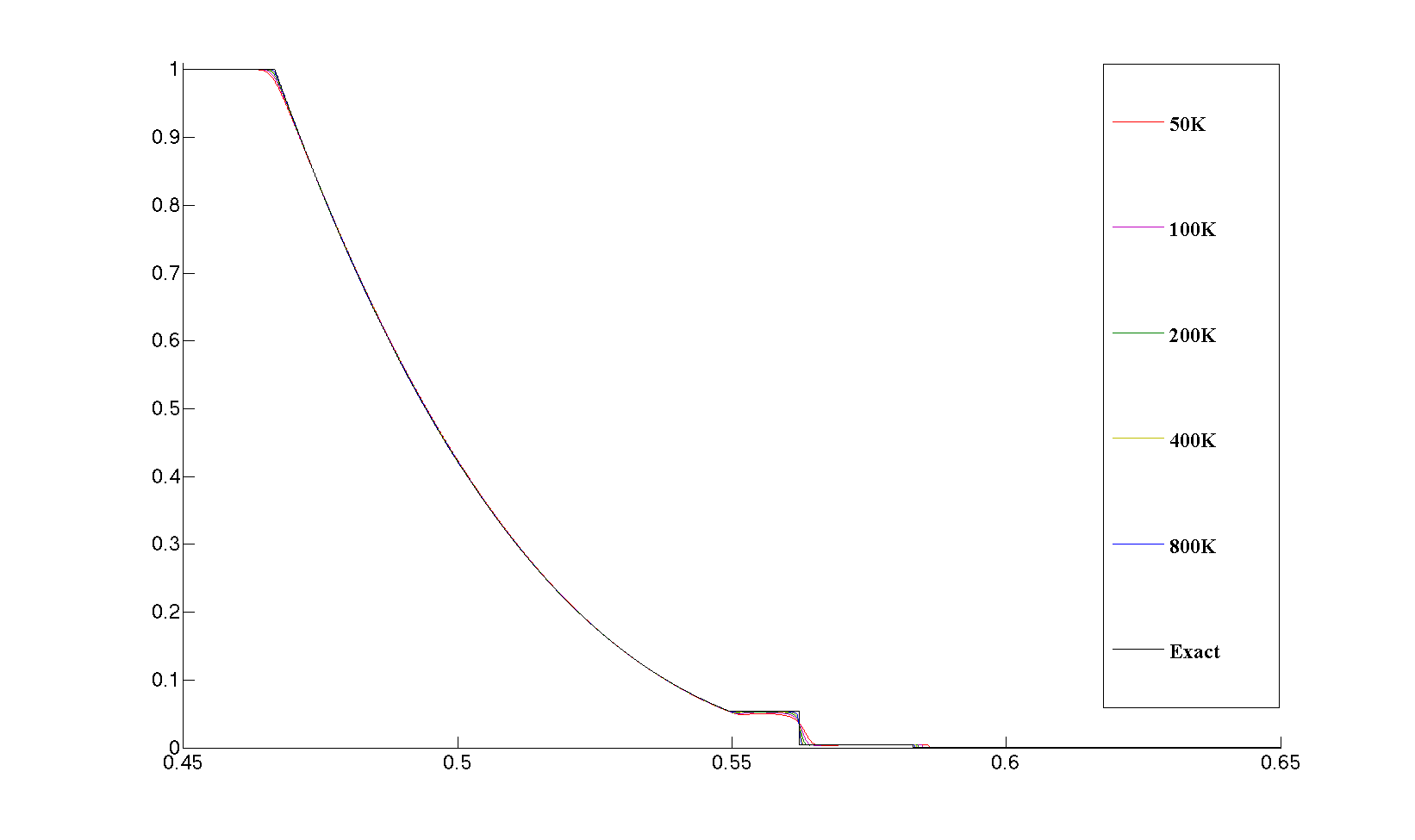}\hfil
\includegraphics[width=0.54\textwidth]{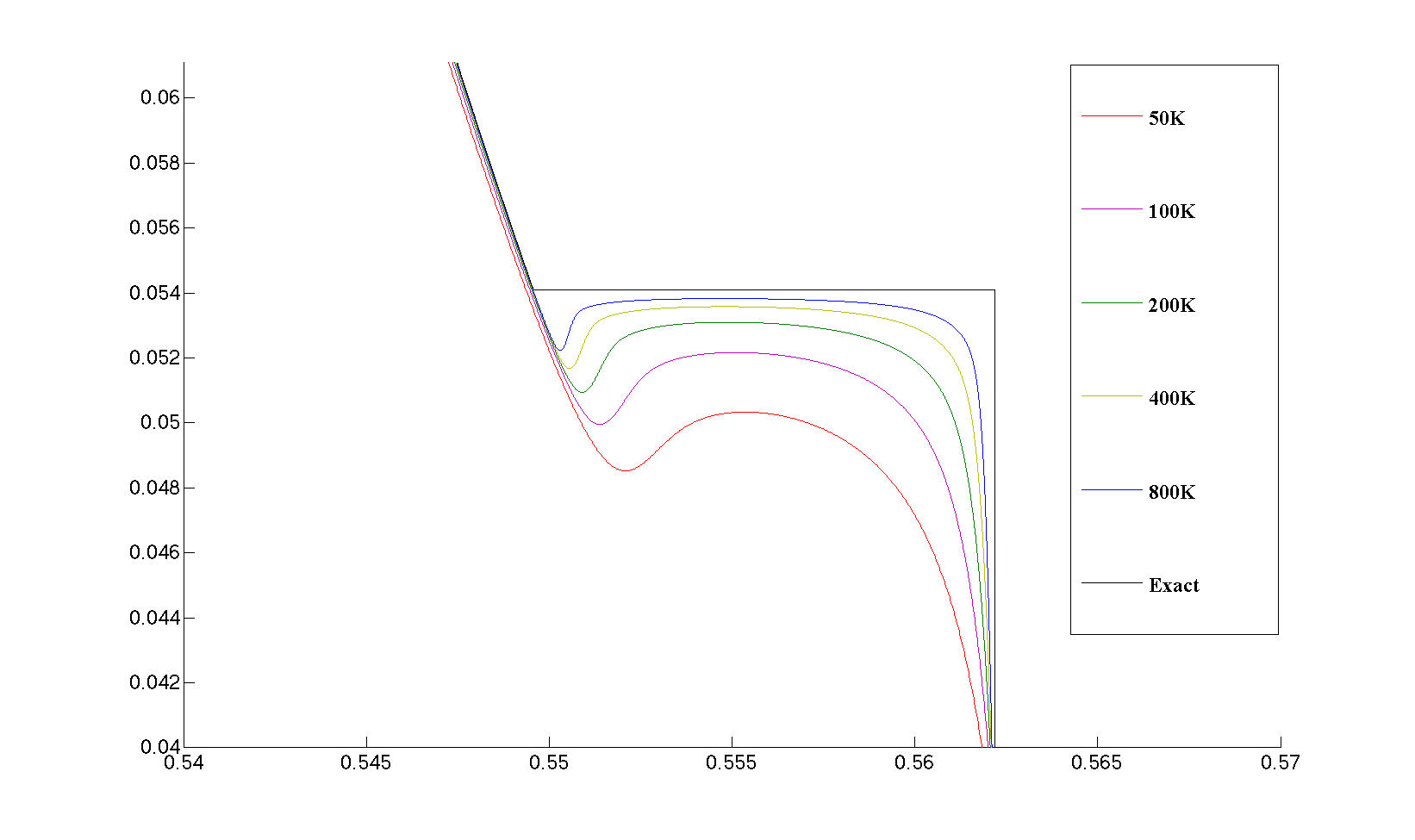}}
\caption{Left: Density profile for the Leblanc Shocktube at $t=0.1$.
  Right: close up view of the density profile at the foot of the
  rarefaction wave.}
\label{Fig:leblanc}
\end{figure}
The structure of the solution is standard; it consists of a
rarefaction wave moving to the left, a contact discontinuity in the
middle and a shock moving to the right. The density profile is
monotone. We solve this problem with the
algorithm~\eqref{def_of_scheme_dij}-\eqref{Def_of_dij} using piecewise
linear finite elements. The density profile computed with $50{,}000$,
$100{,}000$, $200{,}000$, $400{,}000$ and $800{,}000$ grid points is
shown in the left panel of Figure~\ref{Fig:leblanc}.  The right panel
in the figure shows a close up view of the region at the foot of the
expansion wave.  Of course the scheme does not have any problem with
the positivity of the density and the internal energy, but we observe
that the numerical profile is not monotone; there is a small dip at
the foot of the expansion. There is nothing wrong here, since, for
each mesh, the numerical solution is guaranteed by
Theorem~\ref{Thm:UL_is_invariant} to be in the smallest convex
invariant set that contains the Riemann data.  This phenomenon is
similar to what has been observed for the p-system in the previous
section. This example shows again that the invariant domain property
is a different concept than monotonicity, and just looking at
monotonicity is not enough to understand hyperbolic systems.

\section{Concluding remarks}
We have proposed a numerical method to solve hyperbolic systems using
continuous finite elements and forward Euler time stepping. The
properties of the method are based on the introduction of an
artificial dissipation that is defined so that any convex invariant
sets is an invariant domain for the method. The main result of the
paper are Theorem~\ref{Thm:UL_is_invariant} and
Theorem~\ref{Thm:entrop_ineq}. The method is formally first-order
accurate with respect to space and can be made higher-order with respect
to the time step by using any explicit Strong Stability Preserving
time stepping technique. Although, the argumentation of the proof of
Theorem~\ref{Thm:UL_is_invariant} relies on the notion of Riemann
problems, the algorithm does not require to solve any Riemann problem.
The only information needed is an upper bound on the local maximum
speed. Our next objective is to work on a generalization of the FCT
technique (see \cite{KuzminLoehnerTurek2004}) to make the method at
least formally second-order accurate in space and still be domain
invariant.


\end{document}